\documentclass{amsart}
\usepackage{amsmath,amssymb}
\usepackage{graphicx,subfigure}

\usepackage{color}

\newtheorem{theorem}{Theorem}[section]
\newtheorem{proposition}[theorem]{Proposition}
\newtheorem{corollary}[theorem]{Corollary}

\newtheorem{lemma}[theorem]{Lemma}

\theoremstyle{definition}

\newtheorem*{definition*}{Definition}

\theoremstyle{remark}
\newtheorem{remark}[theorem]{Remark}

\numberwithin{equation}{section}

\renewcommand\star{*}

\newcommand{\al}{\alpha}
\newcommand{\be}{\beta}

\newcommand{\ep}{\varepsilon}

\newcommand{\ga}{\gamma}

\newcommand{\la}{\lambda}

\newcommand{\si}{\sigma}
\newcommand{\te}{\theta}

\newcommand{\De}{\Delta}
\newcommand{\Ga}{\Gamma}
\newcommand{\La}{\Lambda}
\newcommand{\Si}{\Sigma}
\newcommand{\Om}{\Omega}

\newcommand{\tq}{\widetilde{q}}

\def\CC{\mathbb{C}}
\def\NN{\mathbb{N}}
\def\RR{\mathbb{R}}

\def\ZZ{\mathbb{Z}}

\renewcommand\SS{\mathbb{S}}
\newcommand{\cA}{{\mathcal A}}

\newcommand{\cC}{{\mathcal C}}

\newcommand{\cE}{{\mathcal E}}
\newcommand{\cF}{{\mathcal F}}

\newcommand{\cH}{{\mathcal H}}

\newcommand{\cK}{{\mathcal K}}
\newcommand{\cL}{{\mathcal L}}

\newcommand{\cP}{{\mathcal P}}

\newcommand{\cV}{{\mathcal V}}

\newcommand\cZ{\mathcal Z}

\newcommand{\fa}{\mathfrak{a}}
\newcommand{\fe}{\mathfrak{e}}
\newcommand{\fcs}{\mathfrak{cs}}
\newcommand{\fb}{\mathfrak{b}}

\newcommand{\pd}{\partial}
\newcommand\minus\backslash

\newcommand\lan\langle
\newcommand\ran\rangle

\newcommand{\supp}{\operatorname{supp}}

\newcommand{\Ker}{\operatorname{Ker}}

\DeclareMathOperator\Div{div} 
\DeclareMathOperator\Real{Re}

\DeclareMathOperator\dist{dist}

\newcommand\DD{\mathbb D}
\renewcommand\leq\leqslant
\renewcommand\geq\geqslant
\newlength{\intwidth}

\DeclareMathOperator\Imag{Im}

\newcommand{\Cone}{\mathbb{C}}

\newcommand{\op}{\operatorname}

\newcommand\bom\varpi
\newcommand{\nablat}{\nabla^{\perp}}
\newcommand{\nablap}{\nabla^{\parallel}}

\newcommand{\sr}{\sqrt{r}}
\newcommand{\srn}{\sqrt{r_n}}

\newcommand\siD{\si_{\DD}}

\begin{document}

\title[Limiting measures and energy growth]{Limiting measures and energy growth\\ for
  sequences of solutions to\\ Taubes's Seiberg--Witten
equations}

\author{Alberto Enciso}
\address{Instituto de Ciencias Matem\'aticas, Consejo Superior de
  Investigaciones Cient\'\i ficas, 28049 Madrid, Spain}
\email{aenciso@icmat.es}

\author{Daniel Peralta-Salas}
\address{Instituto de Ciencias Matem\'aticas, Consejo Superior de
 Investigaciones Cient\'\i ficas, 28049 Madrid, Spain}
\email{dperalta@icmat.es}

\author{Francisco Torres de Lizaur}
\address{Department of Mathematics, University of Toronto, Toronto, ON M5S 2E4, Canada}
\email{ftlizaur@math.toronto.edu}

%
%
\begin{abstract}
We consider sequences of solutions $(\psi_n,A_n)_{n=1}^\infty$ to Taubes's modified Seiberg--Witten equations, associated with a fixed volume-preserving vector field~$X$ on a 3-manifold and corresponding to arbitrarily large values of the strength parameter~$r_n\to\infty$. In Taubes's work, the asymptotic behavior of these solutions is related to the dynamics of~$X$. We consider the rather unexplored case of sequences of solutions whose energy is not uniformly bounded as $n\to\infty$. Our first main result shows that when the energy grows more slowly than $r_n^{1/2}$, the limiting nodal set of the solutions converges to an invariant set of the vector field~$X$. The main tool we use is a novel maximum principle for the solutions with the key property that it remains valid in the unbounded energy case. As a byproduct, in the usual case of sequences of solutions with bounded energy, we obtain a new, more straightforward proof of Taubes's result on the existence of periodic orbits that does not involve a local analysis or the vortex equations. Our second main result proves that, contrary to what happens in the bounded energy case, when the energy is unbounded there are no local restrictions to the limiting measures that may arise in the modified Seiberg--Witten equations. Furthermore, we obtain a connection about the dimension of the support of the limiting measure (as expressed through a $d$-Frostman property) and the energy growth of the sequence of local solutions we construct.
\end{abstract}
\maketitle


\tableofcontents

\section{Introduction}

Taubes's celebrated proof of the Weinstein conjecture in dimension~3 hinges on the analysis of a modified version of the
Seiberg--Witten equations~\cite{T07}, which were originally introduced to
study supersymmetric gauge theories in four dimensions. To define
Taubes's equations, one starts off with a closed oriented 3-manifold~$M$,
endowed with a smooth volume form~$\mu$, and an exact
volume-preserving vector field~$X$ that does not vanish. We recall that $X$ is said to be {\em exact}\/ if $i_X\mu$ is an exact $2$-form. The modified Seiberg--Witten
equations are then a
gauge-invariant semilinear elliptic system on~$M$ that depends on the
vector field~$X$ and on a large parameter~$r$. The gist
of Taubes's approach is to relate the dynamics of the vector field~$X$
with the concentration properties of a certain sequence of solutions as
$r\to\infty$.

Let us record here the form of the system of PDEs considered by
Taubes. For this, one starts by noting that one can take an {\em
  adapted metric}\/, that is, a Riemannian metric~$g$ on~$M$ such
that $\mu$ is the corresponding volume form and~$X$ is a unit vector in
this metric: $g(X,X)=1$. There is no loss of generality in assuming
that~$\mu$ is normalized so that $\int_M\mu=1$. If we now denote by
\[
\la:=g(X,\cdot)
\]
the 1-form dual to the vector field~$X$, Taubes's modified
Seiberg--Witten equations is a system of equations defined using the
metric~$g$ and depending on a real parameter $r\gg 1$. The unknowns are $A$,
which is a connection on a complex line bundle, and $\psi$, which is a
section of a related $\CC^2$~bundle of spinors. The equations read
as
\begin{equation}
\label{SW}
\begin{split}
  *F_{A} &= r(\lambda-\psi^\dagger\si\psi) +
  \bom\,,\\
  D_{A}\psi &= 0\,,
\end{split}
\end{equation}
where $*F_{A}$ denotes the Hodge dual of the curvature 2-form of the
connection~$A$ (which we take to be real valued), $D_{A}$ is the Dirac operator defined by this
connection and the Riemannian metric and $\psi^\dagger \si \psi$ is a
1-form, depending quadratically on the spinor~$\psi$, which is
defined using Clifford multiplication on the spinor bundle. The
equations depend on an auxiliary 1-form~$\bom$ and on
a reference connection~$A_0$, which must be chosen carefully and are bounded in the~$C^3$ norm by a constant independent of~$r$. Precise
definitions will be provided in Section~\ref{S.setting}.

\begin{remark}
For convenience, we follow the usual notation according to which a connection on a complex line bundle is locally written as  $-iA$, so that $A$ is a real 1-form. The curvature $F_A$ is the 2-form written locally as $F_A=dA$. In other words, if $A_{T}$ and $F_{A_T}$ denote the imaginary valued connection and its corresponding curvature (as in Taubes's article~\cite{T07}), we have
$A_{T}=-i A$ and
$F_{A_T}=-i F_{A}$.

\end{remark}

A key quantity in Taubes's analysis of the concentration properties of a
sequence of solutions to the modified Seiberg--Witten equations is the so-called {\em energy}\/ of
the connection~$A$,
\[
\cE(A):= \int_M \la\wedge F_{A}\,.
\]
Although it is not obvious a priori, one can show~\cite{T07} that, for
any solution of the Seiberg--Witten equations, the
energy can be estimated as
\[
-C<\cE(A)< Cr\,.
\]
The first part of Taubes's proof of the Weinstein conjecture in
dimension~3 is to show, in a technical tour
de force building upon the work of Kronheimer and Mrowka~\cite{KM}, that if $X$ is the Reeb field of a contact
form, then one can construct a sequence
$(r_{n},\,\psi_n,\,A_n)_{n=1}^\infty$ of solutions of fixed degree to the modified
Seiberg--Witten equations
with $r_n\to\infty$ and bounded energy (i.e., $\cE(\cA_n)<C$); see~\cite[Section~3]{T07} for a definition of the degree of a solution. The
second part of the proof consists in analyzing the limiting measures
defined by a sequence of solutions with fixed degree and bounded energy.

The state of the
art concerning our knowledge of limiting measures for the
Seiberg--Witten equations is summarized in the following theorem. The
statement uses the {\em
  helicity}\/ of the exact vector field~$X$~\cite{Vogel,PNAS}, which can be written as
\[
  \cH(X):= \int_M \ga\wedge d\ga = \int_M *(\ga\wedge d\ga)\, \mu
\]
in terms of the 1-form~$\ga$ defined by the equation $i_X\mu= d\ga$
modulo a closed 1-form which does not contribute to the integral. Here
$*$~denotes the Hodge star operator.

\begin{remark}
For ease of notation, in the statement of the theorem below and in what follows, we  often identify $3$-forms and their corresponding signed measures in the obvious way: if $\Omega$ is a 3-form then there is a signed measure (which we will denote by $\Omega$ or $d\Om$ when no confusion may arise) defined as
\[
\int_{M} f \, d \Omega:=\int_{M} f \, \Omega=\int_{M} f\, (\star \Omega) \,\mu
\]
for each $f\in C(M)$.
\end{remark}

\begin{theorem}[Taubes~\cite{T07,T09}]\label{T.taubes}
Suppose that the helicity of the vector field $X$ is positive. Then
there exists a sequence $(r_{n},\,\psi_n,\,A_n)_{n=1}^\infty$  of solutions
to the modified Seiberg--Witten equations~\eqref{SW} with
$r_n\to\infty$ and fixed degree. Furthermore:
\begin{enumerate}
\item If the sequence of energies $\cE_n:=\cE(A_n)$ is bounded (i.e.,
  $\cE_n< C$), then the vector field~$X$ possesses at least one
  periodic orbit. 

\item If the sequence of energies is not bounded, the signed measures
\[
\si_n:=\frac{\la\wedge F_{A_n}}{\cE_n}
\]
converge, possibly after passing to a subsequence, to an invariant
probability measure $\sigma_{\infty}$ of $X$. This measure satisfies
\[
  \int_M *(\ga\wedge d\ga)\, d\si_\infty \leq 0\,,
\]
so it is not the volume.
\end{enumerate}
\end{theorem}

\begin{remark}\label{R.=0}
In Theorem~\ref{T.foliated} we will show that the last assertion can be refined to show that, in fact, one can take a subsequence so that
\[
  \int_M *(\ga\wedge d\ga)\, d\si_\infty = 0\,,
\]
provided that the energy growth is sublinear.
\end{remark}

When $X$ is the Reeb field of a contact form, then there exists a sequence of solutions of fixed degree with bounded energy~\cite{T07}. For other kinds of exact volume preserving vector fields, however, all sequences of solutions could have unbounded energy.

One is thus naturally led to the goal of extracting more properties of the invariant measure $\sigma_{\infty}$ in the unbounded energy case. This is an interesting question on geometric analysis and could provide new techniques to study the dynamics of volume-preserving $3$-dimensional vector fields. Despite its promise, there have not been any further developments in this direction, and any other properties of the invariant measures $\sigma_{\infty}$ remain a mystery.

Our objective in this paper is to analyze the limiting measures for
sequences of solutions with unbounded energy. Specifically, we shall next present two theorems which illustrate,
and under suitable hypotheses provide precise statements of, the following two
rough guiding principles:
\begin{enumerate}

  \item The support of the limiting measure is contained in the set where $|\psi_n|$ tends to~$0$ (Theorem~\ref{T.2} and Proposition~\ref{L.con1}).
  \item There are no local obstructions to the limiting measures when the energy is unbounded, so the problem is inherently global (Theorem \ref{T.3}).
\end{enumerate}
Needless to say, we do not expect these principles to hold is all
generality; however, the theorems we state below show that they do
provide useful intuitions. We hope that these results will spark further developments on this subject.

The main difficulty in the unbounded energy case is that, over small scales, the solutions to the Seiberg--Witten equations can no longer be interpreted as approximate solutions to the 2-dimensional vortex equations with finite energy. This asymptotic small-scale behavior is a key ingredient in Taubes's approach.

To overcome this difficulty we resort to a combination of various tools, the most important
of which is a new maximum principle for the Seiberg--Witten
equations (Theorem \ref{T.mprincip}). The key feature of this maximum principle is that it applies no matter if the energy is bounded or not. Although we are mostly interested in the latter case, when the energy is bounded, this provides a substitute of Taubes's local
analysis based on the vortex equations. This enables us to provide a different, more straightforward proof of the corresponding results.

\subsection{Limiting measures supported on the set where ${|\psi_n|\to0}$}

An important observation of Taubes~\cite{T07} (which follows immediately from the bounds in Lemma~\ref{L.TaubesEstimates}) is that
\[
1-|\psi_n|^2 \geq-\frac{C}{r_n}\,,
\]
so for all large~$n$ and any $p\in M$, $|\psi_n(p)|^2$ is bounded by a constant as close to~1 as desired. This does not imply that $|\psi_n|^2$ converges to an indicator function because the smooth functions~$|\psi_n|^2$ can oscillate wildly. However, the way the ``zeros and ones'' of $|\psi_n(p)|^2$ are distributed across the manifold has much to do with the dynamics of the vector field~$X$, and the analysis of the sets where $|\psi_n|^2$ tends to~0 or~1 lies at the very heart of Taubes's proof.

Our first main result shows that when the energy grows slower than $r^{\frac12}$, the set of points of $M$ where $|\psi_n|$ tends to 0 (that is, the limiting nodal set) is invariant under the flow of~$X$. The tools we develop to prove this result provide, in the special case of sequences with bounded energy, a direct proof of Taubes's celebrated periodic orbit theorem (item (i) in Theorem~\ref{T.taubes}), see Section \ref{S.orbits}. Contrary to Taubes's, this proof does not rely on the relationship between the small scale behavior of the Seiberg--Witten and the vortex equations. We want to emphasize that the following theorem is the natural generalization of Taubes's periodic orbit theorem for solutions with unbounded energy.

\begin{theorem}\label{T.2}
Suppose that~$X$ has positive helicity and consider a sequence of solutions $(r_{n},\psi_{n},A_{n})_{n=1}^\infty$ to
the associated Seiberg--Witten equations with $r_n\to\infty$ and $\cE_n= o (r_n^{\frac12})$, i.e.,
\[
\limsup_{n\to\infty} \frac{\cE_n}{r_n^{1/2}}=0\,.
\]
Then:
\begin{enumerate}
\item For any fixed  $\theta \in (0, 1)$, the set
\[
Z^{\theta}_n:=\left\{p \in M: 1-|\psi_n|^2(p) \geq \theta \right\}
\]
is non-empty for $n$ large enough, and any convergent subsequence (in the Hausdorff metric) converges to a closed subset $Z^\theta_{\infty}$ which is invariant under the flow of $X$.

\item The collection of limiting sets $Z^\theta_{\infty}$ is independent of $\theta$, in the sense that, for any converging subsequence $Z^\theta_n$, the corresponding subsequence $Z_n^{\theta'}$, for any $\theta'\neq\theta$, is also converging with the same limit, i.e., $Z^{\theta'}_\infty=Z^\theta_\infty$.

\item There is a constant $C$, independent of $n$, such that any convergent subsequence of sets

\[
Z_n:=\left\{p \in M: |\psi_n|^2(p) \leq C\max (r_n^{-\frac{1}{4}},\cE_n r_n^{-\frac{1}{2}}) \right\}
\]
also converges to an invariant set $Z_{\infty}$. The collection of such limiting sets coincides with the limiting sets $Z^\theta_\infty$ in the sense specified above.
\end{enumerate}
In the case that the sequence of energies is uniformly bounded, the invariant set $Z_\infty$ consists of a finite collection of periodic orbits of $X$.
\end{theorem}

\begin{remark}
If instead of taking the Hausdorff limit of the sequences of sets in Theorem~\ref{T.2}, we take the upper Kuratowski limit, this is always compact and unique (independent of the subsequence), so we can write $Z^\theta_\infty=Z_\infty$ for all $\theta\in(0,1)$.
\end{remark}

This theorem is proved in Section \ref{S.T2}, using the maximum principle presented in Section \ref{S.max}.
It should be emphasized that, in general, the limiting invariant set $Z_{\infty}$ could be the whole manifold $M$. Indeed, because of the high oscillations of $|\psi_n|$ for large $n$, the fact that a point $p$ is in $Z_{\infty}$ does not imply that $1-|\psi_n|^2(p) >0$ for all large enough $n$; it could very well happen that $|\psi_n|(p)=1$ for all $n$. We can only characterize $Z_\infty$ when the energy is uniformly bounded.

Concerning points where $|\psi_n|$ tends to~1, the next proposition establishes that if $|\psi_n|\to 1$ on an open set~$U$, then the limiting measure does not charge this set. This result is proved in Section~\ref{S.open1}.

\begin{proposition}\label{L.con1}
Let $(r_n,\psi_n, A_n)_{n=1}^\infty$ be a sequence of solutions with unbounded energy. If $|\psi_n|\to 1$ pointwise on an open set $U\subset M$ as $n\to\infty$, then $\si_\infty(U)=0$.
\end{proposition}

\subsection{Absence of local obstructions for the limiting measures}

Our second main result proves that, locally, any invariant measure can arise as the limiting measure for some sequence of solutions to the Seiberg--Witten equations. Thus, contrary to what happens in the bounded energy case, when the energy is unbounded any attempt to derive some restrictions to the possible invariant sets of the vector field~$X$ from the PDE must involve global arguments.

To state the theorem, we start by
fixing a flow box
$\cC\subset M$ of the vector field~$X$. We choose local coordinates and identify $\cC= (0,1)\times\DD$, where
$\DD$ is the (open) unit 2-dimensional disk, and assume that $X=\pd_t$ with
$t$ being the coordinate on the interval~$(0,1)$. Note that any
$X$-invariant measure on~$\cC$ can then be written as
\[
\sigma=\siD \otimes dt \,
\]
where $\siD$ is a measure supported on~$\DD$ and $dt$ is the Lebesgue measure on
the interval. Without loss of generality, we can
normalize~$\siD$ and assume that it is a probability measure.

The following theorem does not only show that there are no local obstructions for the limiting measure obtained from solutions to the Seiberg--Witten equations. Furthermore, it also suggests that there is a connection between the dimension of the support of the invariant measure and the energy sequence. Roughly speaking, the faster the growth of $\cE_n$ that we allow, the larger the dimension of the support of the measure~$\si_\infty$. A convenient way of articulating this connection is by recalling that a probability measure $\sigma$ on~$\DD$ is {\em $d$-Frostman}\/ if the measure of any ball $B(x,\ep)$ of radius~$\ep$ is bounded as
\[
\sigma(B(x,\ep))\leq C\ep^d\,,
\]
for all $\ep>0$ and $x\in\DD$. It is standard~\cite[Exercise~1.15.20]{Tao} that this property implies that the Hausdorff dimension of the support of $\sigma$ is at least $d$ (i.e., $\dim_H \text{(supp } \sigma)\geq d$), but this property is strictly stronger in that it provides some quantitative control on the measure. It is worth mentioning that the dimension of the support of the metric is also connected with its regularity (i.e., very roughly speaking, the better the integrability properties of the weak derivatives of the measure, as estimated using Sobolev or Besov spaces, the higher the Hausdorff dimension of its support). For the benefit of the reader, we specify this connection in Proposition~\ref{P.frost}.

\begin{theorem}\label{T.3}
  Let $\siD$ be any probability measure on~$\DD$. There is an adapted metric~$g$ on~$\cC$ and a sequence of solutions $(r_n,\psi_n,A_n)$ to the Seiberg--Witten
  equations~\eqref{SW} with $\bom=0$ on the flow box $\cC$ such that
 \begin{enumerate}
 \item Setting $\cE_n:= \int_{\cC} \la\wedge F_{A_n}$, we have
  \[
\frac{\la\wedge F_{A_n}}{\cE_n}\rightarrow \siD\otimes\, dt
\]
in the sense of weak convergence of measures.
\item If $\sigma_{\DD}$ is $d$-Frostman for some $d>0$, then we can choose the sequence of solutions such that
\[
\lim_{n\to\infty}\cE_{n} r_{n}^{-\theta}= 0
\]
with $\theta:=\min\big\{\frac{1}{4}, \, \frac{d}{2(d+1)}\big\}$.
\end{enumerate}
\end{theorem}

\subsection{Structure of the paper}

In Section~\ref{S.setting} we recall the definition of the various
objects appearing in the modified Seiberg--Witten equations and some properties of the solutions. Some further auxiliary
equations are derived too. In Section~\ref{S.max} we prove a new
maximum principle for these equations that can be effectively applied
to sequences of solutions with unbounded energy. This result turns out to be a fundamental tool to analyze the properties of the limiting invariant measures. The proofs of Theorem~\ref{T.2} and Proposition~\ref{L.con1} are presented in
Section~\ref{S.T2}. As an additional application of the new maximum principle, we also include an alternative proof of Taubes's periodic orbit theorem. The proof of Theorem~\ref{T.3} on the absence of local obstructions for the limiting measure is given in Section~\ref{S.3}. Finally, in Section~\ref{S.ergodic} we show that the vector field $X$ cannot be ergodic provided that the energy growth is linear.  We also include Appendix~\ref{S.foliated} with an additional result that can be useful for future work in the subject. Concretely, we reinterpret the concentration properties of solutions to the Seiberg--Witten equations using Sullivan's theory of currents, thus implying as a particular consequence the refinement stated in Remark~\ref{R.=0}.

\section{Setting and preliminary results}\label{S.setting}

In this section, following Taubes~\cite{T07} (see also~\cite{Sun}), we include the precise formulation of the modified Seiberg--Witten equations, Taubes's theorem on the existence of solutions (Theorem~\ref{T.existence}) and the basic a priori estimates (Lemmas~\ref{L.TaubesEstimates} and~\ref{L.bounds2}). Our main contribution is Proposition~\ref{P.albe}, which shows that the function $|\alpha|^2$ defined below satisfies an explicit second order elliptic PDE on $M$. This statement is not included in Taubes's works and is key to prove the new maximum principle we present in Section~\ref{S.max}.

\subsection{Definitions and existence of solutions}\label{S.defs}

Let us recall the definition of the modified
Seiberg--Witten equations. The reader can find further details
in~\cite{T07,T09,Hutchings}. Throughout, $g$ denotes a fixed
Riemannian metric on the 3-manifold $M$ adapted to the volume-preserving
vector field~$X$, and $\la:= g(X,\cdot)$ stands for the dual
1-form. In
what follows, we will assume that the helicity of the vector field~$X$
is positive.

We start by recalling that a {\em spin-c structure\/} on $M$ is a
pair $\frak s= ({\mathbb
    S},\si)$, where ${\mathbb S}$ is a rank-$2$ Hermitian vector
  bundle on $M$, called the {\em spinor bundle\/}, and
\[
\si: TM \longrightarrow \op{End}({\mathbb S})
\]
is a bundle map, called {\em Clifford multiplication\/}, such that for each $p\in M$:
\begin{enumerate}
\item
If $V,W\in T_pM$, then
\[
\si(V)\si(W) + \si(W)\si(V) = -2g(V,W)\,.
\]
\item
If $e_1,e_2,e_3$ is an oriented orthonormal frame for $T_pM$,
then
\[
\si(e_1)\si(e_2)\si(e_3)=1\,.
\]
\end{enumerate}
Of course, taking a local trivialization, we can identify $\SS$ with $\CC^2$.

A {\em spin-c connection} $A$ on $\SS$ is a connection that behaves naturally with respect to the Clifford multiplication: for any section $\psi: M \rightarrow \SS$ and any vector fields $X, Y \in TM$, we have
\[
\nabla^{A}_X (\sigma(Y) \psi)=\sigma(\nabla^{LC}_{X} Y) \psi+\sigma(Y)\nabla^{A}_{X} \psi
\]
where $\nabla^{LC}$ is the Levi-Civita connection on $TM$ induced by the metric $g$.

Consider the oriented 2-plane field $K^{-1}:=\Ker \la$,
which we regard as a Hermitian line bundle. It is standard that
$K^{-1}$ determines a distinguished spin-c structure $\frak s_\la=(\SS_\la,\si_\la)$ on~$M$, in which
\begin{equation}\label{splitting0}
{\mathbb S}_\la=\underline{\CC}\oplus K^{-1}\,,
\end{equation}
where $\underline{\CC}$ denotes the trivial complex line bundle over~$M$.
Any spin-c structure $\frak s =(\SS,\si)$ is obtained from this one by
tensoring with a suitable Hermitian line bundle $E$, so that
\begin{equation}
\label{eqn:SEKE}
{\mathbb S} = E \otimes {\mathbb S}_\la= E\oplus K^{-1}E\,
\end{equation}
and $\sigma=\sigma_{\lambda}\otimes \text{Id}_{E}$.
In this decomposition, $E$ is the $+i$ eigenspace of Clifford
multiplication by the vector field dual to the 1-form $\lambda$, while $K^{-1}E$ is the $-i$
eigenspace.

In what follows, $E$ stands for a fixed line bundle whose first Chern
class is such that $c_1(K^{-1})+ 2 c_1(E)$ is torsion
in~$H^2(M;\ZZ)$. This is equivalent to requiring
that $\det(\SS)$ has torsion first Chern class, since any spin-c connection on $\SS$ induces a connection on the determinant bundle
$\det(\SS)$ that can be written as $A_0+2A$, where $A_0$ is a connection on $K^{-1}$ (inducing a natural spin-c connection on $\SS_{\lambda}$) and $A$ is a connection on $E$.

The unknowns in the Seiberg--Witten equations are a
spinor~$\psi$, which is a section of~$\SS$, and a connection~$A$
on~$E$, whose curvature we denote by $F_A$. We also need an auxiliary connection~$A_0$ on~$K^{-1}$, which
we pick (following Taubes) as the only connection such that
\begin{equation}\label{fiducial}
D_{A_0}\psi_0=0\,,
\end{equation}
where $\psi_0:=(1,0)$ is a section of the distinguished spin
bundle~$\SS_\la$, and $D_{A_0}$ should be understood as the Dirac operator associated with the unique spin-c connection on $\SS_\lambda$ defined by the connection $A_0$ on $K^{-1}$ and the trivial connection on $\CC$. It is well known that the connections
$A_0$ on~$K^{-1}$ and $A$ on~$E$ determine a unique spin-c connection on~$\frak s$, which we will
denote by~$\nabla_A$; the associated Dirac operator is then defined as $D_A:=\sigma(\nabla_A)$. Taubes's modified Seiberg--Witten equations read
\begin{equation}
\label{SWrepeat}
\begin{split}
  *F_{A} &= r(\lambda-\psi^\dagger\si\psi) +
  \bom\,,\\
  D_{A}\psi &= 0\,,
\end{split}
\end{equation}
where $r>1$ is a real parameter. Here $\bom$ is a given 1-form whose significance will become clear in a moment, and $\psi^\dagger\si\psi$ is the 1-form that acts on any vector field $V$ as:
\[
\psi^\dagger\si\psi(V):=-i \psi^\dagger\si(V)\psi \,.
\]
Notice that the properties of the Clifford map $\sigma$ ensure that the 1-form $\psi^\dagger\si\psi$ is real valued.

\begin{remark}\label{R:local}
In local coordinates on a ball $B\subset M$, if $\{e_1,e_2,e_3\}$ is a local orthonormal frame so that $\{e_2,e_3\}$ span $\Ker \la$, then $\psi$ is a function from $B$ to $\CC^2$, $A$ and $A_0$ are (real-valued) $1$-forms,
$$\sigma(X)=i\Big[(X\cdot e_1)\sigma_1+(X\cdot e_2)\sigma_2+(X\cdot e_3)\sigma_3\Big]\,,$$
where $\sigma_k$ are the Pauli matrices, and the covariant derivative $\nabla_A \psi$ can be understood as two complex-valued vector fields on $B$ given by
\[
\nabla_A \psi=\nabla \psi+  \Lambda  \psi- \frac{i}{2} (A_0^{\sharp}+2 A^{\sharp})  \psi\,.
\]
Here  $A^\sharp$ is the vector field associated with the 1-form~$A$ and $\Lambda$ is the $2 \times 2$ matrix-valued vector field given by
 \[
 \Lambda= \frac{1}{8} g(\nabla_{e_j} e_{m}, e_{n}) [\sigma(e_n), \sigma(e_m)] e_{j}\,.
\]
Summation over repeated indices is understood.
\end{remark}

We are now ready to state the fundamental existence theorem due to
Taubes~\cite{T07} that we will need in this paper. A caveat is
that we have not defined what one means by the {\em degree}\/ of the solutions to
the modified Seiberg--Witten equations whose existence is proved
here. The notion of degree can be defined using the Seiberg--Witten--Floer
homology but, since we will not need it in the following, we refer
to~\cite{T07,Hutchings} for the precise definition. We stress that in the case of Reeb fields, the value of the energy (which is always finite) can be related to the degree~\cite{Sun}.

Let us also
record here that a solution $(A,\psi)$ is called {\em irreducible}\/
if $\psi$ is not identically~0. Finally, we will denote by $\bom_K$ the harmonic
1-form on~$M$ with the property that the Hodge dual $*\bom_K$
represents the image in the cohomology group $H^2(M)$ of the first
Chern class~$c_1(K)$. Equivalently, one can set
\begin{equation}\label{bom}
\bom_K:=-\frac{1}{2}* F_{A_0}+ * d \Ga\,,
\end{equation}
where the $1$-form $\Ga$ satisfies
\[
d* d \Ga=\frac{1}{2}d * F_{A_0}\,.
\]

\begin{theorem}[Taubes]\label{T.existence}
Let $X$ be a nonvanishing volume-preserving vector field with positive
helicity. There is a real number $0<\theta<1$ and an infinite set of
negative integers $\cK$ such that, for each fixed $k \in \cK$, we have:

\begin{enumerate}

\item There exists a smooth 1-form $\bom'$, of arbitrarily
  small $C^3$ norm, such that the Seiberg--Witten Equations~\eqref{SWrepeat} with
  $\bom:=\bom_K+\star d \bom'$ has an irreducible solution~$(\psi,A)$ of degree~$k$ provided
  that the value of the parameter~$r$ belongs to a certain increasing sequence
  $(r_n)_{n=1}^\infty\subset (1,\infty)$ (depending on~$k$) without any accumulation points.

\item The aforementioned sequence of solutions $(\psi_n,A_n)$ of degree~$k$
  corresponding to the value of the parameter~$r_n$ satisfy the
  uniform bound
\begin{align*}
  \|1-|\psi_n|^2\|_{L^\infty(M)}>\te\,.
\end{align*}
\end{enumerate}

\end{theorem}

\subsection{A priori estimates and a useful equation}

Let us henceforth employ the shorthand notation
\[
\psi=:(\al,\be)
\]
for the decomposition according to the splitting~\eqref{eqn:SEKE} of
the spinor part of the solution $(A,\psi)$ to the modified
Seiberg--Witten equations. In view of Remark~\ref{R:local}, it is clear that both $\al$ and $\beta$ can be locally understood as complex-valued functions.

In what follows let $(r_n,\psi_n,A_n)$ be a sequence of solutions as in Theorem~\ref{T.existence} (see also Theorem~\ref{T.taubes}). For future reference we record here an identity connecting
the signed measures $\si_n$ and the decomposition $\psi_n=(\al_n,\be_n)$ that will be useful in the case when
$\cE_n\to\infty$:
\begin{equation}\label{sigma_n}
  \si_n= \frac{    r_n(1-|\al_n|^2)\, \mu}{\cE_n}+ O(\cE_n^{-1})\,\mu\,.
\end{equation}
This follows easily from the second estimate in Lemma \ref{L.TaubesEstimates} and the fact that
\[
(\star \psi^\dagger\si\psi) \wedge \lambda=\psi^\dagger\si\psi (X)\,\mu=(|\al|^2-|\be|^2)\,\mu \,.
\]

Now we recall Taubes's a priori estimates for solutions of the
Seiberg--Witten equations~\cite[Lemmas 2.2 and 2.3]{T07}.  With a slight abuse of notation, here and in what follows we will use $\nabla_{A}$ to denote both the covariant derivative defined by the connection $A$ on $E$ and the covariant derivative defined by $A$ and $A_0$ on $E\otimes K^{-1}$. In other words, in local coordinates $\nabla_A\alpha=\nabla\alpha-iA^\sharp\alpha$ and $\nabla_A\beta=\nabla\beta -i(A^\sharp+A_0^\sharp)\beta$.

\begin{lemma}\label{L.TaubesEstimates}
There exists a constant $C$ such that
the solution $(\psi,A)$ is bounded as
\begin{align*}
|\alpha| & \le 1 + \frac{C}{r}\,,\\
|\beta|^2 &\le \frac{C}{r}\left|1-|\alpha|^2\right| + \frac{C}{r^2}\,,\\
|\nabla_{A} \alpha| & \le C \sr\,,\\
|\nabla_{A} \beta| & \le C\,, \\
|\nabla^2_{A} \alpha| & \le C r\,,\\
|\nabla^2_{A} \beta| & \le C \sr\,.
\end{align*}
In particular, the negative part of $1-|\al|^2$ is bounded as
\[
1-|\al|^2\geq -\frac Cr\,.
\]
\end{lemma}

A first refinement of these a priori estimates we need, which is implicit in Taubes's
work, is a set of anisotropic estimates that provide finer control of
some geometric quantities. To emphasize this anisotropy, it is
convenient to introduce
some further notation. Given a scalar function~$f$ on~$M$, we let
\begin{align*}
\nablap f&:=(X\cdot \nabla f) X \,,\\
\nablat f& := \nabla f- \nablap f
\end{align*}
denote the components of its gradient that are parallel and
perpendicular to the field~$X$, respectively. For sections of the vector bundles $\SS$, $E$ and $E \otimes K^{-1}$, $\nablap_{A} $ and $\nablat_{A}$ are defined analogously.

\begin{lemma}\label{L.bounds2}
A solution~$(\psi,A)$ to the modified Seiberg--Witten equations
satisfies the following anisotropic bounds:
\begin{align*}
  |\nablap |\al|^2| & \leq C\,,\\
  |\nablat |\al|^2| & \leq C\sr\,,\\
  \big|\nabla \nablap |\al|^2\big| & \leq C\sr\,,\\
  |\nabla^2 |\al|^2|&\leq Cr\,,\\
  |X\cdot\nabla_A\al|& \leq C\,.
\end{align*}
\end{lemma}

\begin{proof}
Since $A$ is a Hermitian connection on $E$, we have, for any vector field $V$,
\begin{equation*}
\label{lb1}
V \cdot \nabla |\al|^2= 2\Real( \bar\al V \cdot \nabla_{A} \alpha)\,,
\end{equation*}
so we readily get
\begin{equation*}
\label{lb2}
|\nabla |\al|^2| \leq 2|\al| |\nabla_A \al|\,.
\end{equation*}
Similarly, since
\[
W\cdot\nabla(V \cdot \nabla |\al|^2)= 2 \Real (\overline{W \cdot \nabla_A \alpha} \, V\cdot \nabla_A \alpha)+ 2 \Real[\bar\al W \cdot \nabla_A ( V\cdot \nabla_A \alpha) ]
\]
for any vector fields~$V,W$, one finds that
\begin{equation*}\label{lb3}
|\nabla^2 |\al|^2| \leq 2|\nabla_A \al|^2+ 2|\al| |\nabla^2_{A} \al|\,.
\end{equation*}
These equations together with the bounds in Lemma
\ref{L.TaubesEstimates} automatically imply the second and fourth
estimates we aimed to prove.

To derive the other bounds, we first observe that the Dirac equation implies (see e.g.~\cite[Equation~3.7]{T09})
\[
|X\cdot \nabla_A\al |\leq C (|\nabla_A \beta|+|\beta|+|\alpha|)\,.
\]
Combining all the previous estimates together with the bounds on the derivatives of $\be$ in Lemma~\ref{L.TaubesEstimates}, the remaining inequalities follow.
\end{proof}

Finally, we are ready to state the main result of this section. In the following proposition we show that the function $|\alpha|^2$ satisfies an explicit second order elliptic PDE on $M$. This result will be instrumental in the proof of a new maximum principle for solutions of the Seiberg--Witten equations, cf. Theorem~\ref{T.mprincip}.

\begin{proposition}\label{P.albe}
  The absolute value of~$\al$ satisfies the equation
\begin{equation}\label{eqal2}
|\al|^{2}\Delta |\al|^{2}-  {|\nabla |\al|^2|^2}+2r|\al|^{4}(1-|\al|^{2}-|\be|^2)
  =H(\psi,A)\,
\end{equation}
where the term $H(\psi,A)$ is pointwise bounded as
\begin{align*}
|H(\psi,A)|&\leq C\Big(1+|\nablat |\al|^2|\Big)\
\end{align*}
for an $r$-independent constant $C$.
\end{proposition}

\begin{proof}
As proved in~\cite[Section 6.1]{T07}, $|\al|^2$ satisfies the equation
\[
|\alpha|^2\Delta |\al|^2-2|\al|^2|\nabla_{A} \al|^2+2r|\al|^4(1-|\al|^2-|\beta|^2)=G\,,
\]
where $G$ has the form
\[
G:=-|\al|^2\Big(\tau(\alpha, \beta)+\tau(\alpha, \nabla_A \beta)+\tau(\alpha, \alpha)\Big)
\]
(the notation $\tau(\cdot, \cdot)$ will henceforth represent bilinear maps that only depend on the metric, and that may change from line to line).

Observe that, by virtue of Lemma \ref{L.TaubesEstimates}, we have the pointwise bound $|G|\leq C$. Thus, to prove the proposition it suffices to show that
\[
2|\al|^2|\nabla_A \al|^2= |\nabla |\al|^2|^2+G'\,,
\]
where $G'$ is some function of $A, \psi$ and its derivatives that satisfies the bound
$$|G'| \leq C(1+|\nablat |\al|^2|)\,.$$

First, we notice that
\begin{equation}\label{auxi}
|\al|^2 |\nabla_A \al|^2=|\Real ( \bar\al \nabla_A \al)|^2+|\Imag (\bar\al \nabla_A \al) |^2=\frac{1}{4} |\nabla |\al|^2|^2+|\Imag (\bar\al \nabla_A \al) |^2 \,.
\end{equation}

Next, for convenience we introduce some local notation. For each point~$p\in M$, one can pick two vector
fields $e_1^p,e_2^p$, defined on a neighborhood $V^p$ of~$p$, such that
$\{X(q),e_1^p(q),e_2^p(q)\}$ is an oriented orthonormal basis of the
tangent space $T_qM$ for any $q\in V^p$. For the ease of notation, we
will henceforth omit the superscript~$p$. The vectors $\{e_1(q), e_2(q)\}$ span the transverse distribution
$\Ker  \lambda$ at each point $q \in V$. We also denote by $J(q)$ the
almost complex structure on this 2-plane field defined at~$q$ by
\[
  J(e_1(q)):=e_2(q)\,,\qquad J(e_2(q)):=-e_1(q)\,.
\]
This almost complex structure does not depend on the particular choice
of orthonormal vector fields and is well defined globally on
$\Ker\la$.

Since the complex structure $J$  on $ \Ker\lambda$ preserves the scalar product, Equation~\eqref{auxi} can be rewritten as
\begin{equation}\label{auxx}
|\al|^2 |\nabla_A \al|^2=\frac{1}{4} |\nabla |\al|^2|^2+|\Imag \bar\al \nablap_A \al|^2+|\Imag (J \bar\al \nablat_A \al) |^2\,.
\end{equation}

Now the crucial observation is that one can infer from the Dirac equation $D_{A} \psi=0$ that
\begin{equation}\label{dirackey}
J i \nablat_{A} \alpha=\nablat_{A} \alpha+\Theta \nablap_{A} \beta+\Theta \beta \,.
\end{equation}
(Here and in what follows, we will use $\Theta$ to represent linear maps between the corresponding bundles that depend only on the metric.) Indeed, on the one hand, on the local frame $\{X, e_1, e_2\}$ we have
\[
J i \nablat_{A} \alpha=(i e_1 \cdot \nabla_A \al)e_2-(i e_2 \cdot \nabla_A \al) e_1 \,,
\]
and on the other hand, the Dirac equation implies the following relation between the derivatives of $\al$ and $\beta$ (see e.g.~\cite[Equation~3.7]{T09})
\begin{equation}
X \cdot \nabla_A\be = -i e_1\cdot\nabla_A\al + e_2\cdot \nabla_A\al+\Theta \be\,.
\end{equation}
Using this relation to write $e_1 \cdot \nabla_A \al$ in terms of $e_2 \cdot \nabla_A \al$, and viceversa, we readily get Equation~\eqref{dirackey}.

This understood, we can write:
\begin{align*}
\Imag (J \bar\al \nablat_A \al)=-\frac{1}{2} J \big(i\bar\al \nablat_A \al-i (\overline{\nablat_A \al}) \al \big)=-\frac{1}{2}\big( \bar\al (J i \nablat_A \al)+(\overline{J i \nablat_A \al}) \al \big)=\\=-\Real(\bar\al \nablat_A \al)-\tau(\alpha, \nablap_A \beta)-\tau(\alpha, \beta)\,,
\end{align*}
where we have used Equation~\eqref{dirackey} in the last equality.

Recalling that $\Real (\bar\al \nablat_A \alpha)= \frac{1}{2} \nablat |\al|^2$, we obtain
\[
|\Imag (J \bar\al \nablat_A \al) |^2=\frac{1}{4}|\nablat |\al|^2|^2+\nablat |\al|^2 \cdot \big( \tau(\alpha, \nablap_A \beta)+\tau(\alpha, \beta) \big)+|\tau(\alpha, \nablap_A \beta)+\tau(\alpha, \beta)|^2\,.
\]
Plugging this identity into Equation~\eqref{auxx} we easily infer that
\begin{align*}
|\al|^2 |\nabla_A \al|^2=\frac{1}{4} |\nabla |\al|^2|^2+\frac{1}{4}|\nablat |\al|^2|^2+|\Imag \bar\al \nablap_A \al|^2+\nablat |\al|^2 \cdot \big( \tau(\alpha, \nablap_A \beta)+\tau(\alpha, \beta) \big)\\+|\tau(\alpha, \nablap_A \beta)+\tau(\alpha, \beta)|^2\,,
\end{align*}
so substituting $|\nablat |\al|^2|^2$ by $|\nabla |\al|^2|^2-|\nablap |\al|^2|^2$
we finally obtain
\begin{align*}
|\al|^2 |\nabla_A \al|^2=\frac{1}{2}|\nabla |\al|^2|-\frac{1}{4}|\nablap |\al|^2|^2+|\Imag \bar\al \nablap_A \al|^2+\nablat |\al|^2 \cdot \big( \tau(\alpha, \nablap_A \beta)+\tau(\alpha, \beta) \big)\\+|\tau(\alpha, \nablap_A \beta)+\tau(\alpha, \beta)|^2\,.
\end{align*}
The proposition then follows taking into account the bounds in Lemmas~\ref{L.TaubesEstimates} and~\ref{L.bounds2}.
\end{proof}

\section{A maximum principle for solutions with unbounded energy}
\label{S.max}

In this section we prove a new maximum principle for solutions of the modified Seiberg--Witten equations. Specifically, we establish a dichotomy for the large~$r$ behavior of local minima of~$|\al|^2$ on small disks: either they are close to $0$ or close to $1$ as $r\to\infty$. The main consequence of this result is Theorem~\ref{T.pt} below, which is instrumental in the proof of Theorem~\ref{T.2}. We stress that all the
constants appearing in this section are independent of~$r$. In what follows, $(r,\psi,A)$ is a sequence of solutions as in Taubes's Theorem~\ref{T.existence}, and we recall that $\psi=(\alpha,\beta)$.

\begin{theorem}\label{T.mprincip}
  Let $\rho:(1,\infty)\to (0,1)$ be any continuous function with
  \[
\lim_{r\to\infty}\rho(r)=0\,.
  \]
Suppose that $\Sigma$ is a disk of radius $\rho(r)$, embedded in $M$,
transverse to the vector field~$X$
and perpendicular to it at some point $p$. If a point $q \in \Sigma$ is a local minimum of the restriction of
$|\al|^2$ to~$\Si$, then either
\[
|\al (q)|^2 \leq C\eta(r)^{1/2 }
\]
or
\[
\big||\al (q)|^2-1\big| \leq C\eta(r) \,.
\]
Here
\[
\eta(r):= r^{-1/2}+\rho(r)^2
\]
is another continuous function that tends to~$0$ at infinity.
\end{theorem}

\begin{proof}

Let us start by noticing that the Laplacian (on~$\Si)$ of the restriction of a
scalar function~$f$ to the surface~$\Si$ , which we denote by $\De_\Si
f$, and the restriction to~$\Si$ of the Laplacian of~$f$ are related
through the following  formula:
\begin{align*}
  \De f|_\Si&= \sum_{j=1}^2 \nabla^{2} f (V_j, V_j) + \nabla^{2} f(N, N)\\
  &= \De_\Si f + N \cdot \nabla( N \cdot \nabla f)+ \sum_{j=1}^2 (\Div
    V_j -\Div_\Si V_j)\, V_j\cdot\nabla f + (\Div N) N \cdot \nabla f\,.
\end{align*}
Here $\{V_1,V_2, N\}$ is a local orthonormal basis of
the tangent space of~$M$ chosen so that $N$ is perpendicular to $\Sigma$ at every point. Further, $\Div_\Si V_j$
denotes the divergence of the vector field~$V_j$ (which is
tangent to~$\Si)$ with respect to the induced area form on~$\Si$, $i_{N} \mu$.

If the point $q$ is a local minimum of the restriction of $|\al|^2$ to~$\Si$,
it follows that, at~$q$,
\[
\De_\Si |\al|^2\geq0\,,\qquad \nabla_{\Sigma} |\al|^2=0\,,
\]
where the gradient on~$\Si$ is
\[
\nabla_{\Sigma} |\al|^2= \nabla |\al|^2-(N\cdot \nabla |\al|^2)\, N\,.
\]
Accordingly, the fact that $\nabla_\Si|\al|^2(q)=0$ implies
\[
\nabla |\al|^2(q)= (N\cdot \nabla |\al|^2)\, N(q)=(N-X) \cdot \nabla |\al|^2(q)+X \cdot \nabla |\al|^2(q)\,.
\]
In then follows from the a priori estimates in Lemma~\ref{L.bounds2} and the obvious bound $\|X-N\|_{L^\infty(\Sigma)} < C\rho(r)$, that
\[
|\nabla |\al|^2(q)|\leq C[1+\sr\,\rho(r)]\,.
\]

In view of the formula for~$\De |\al|^2|_\Si$, and
using again that $\|X-N\|_{L^\infty(\Sigma)} < C\rho(r)$, we
infer that, always at the point~$q$,
\begin{align*}
  \De |\al|^2 (q)&\geq  N \cdot \nabla( N \cdot \nabla  |\al|^2)+ (\Div N) N \cdot \nabla  |\al|^2\\
                      &= (N-X) \cdot \nabla((N-X)\cdot \nabla
                        |\al|^2) +  X \cdot \nabla((N-X)\cdot \nabla
                        |\al|^2) \\
  &\qquad+  (N-X) \cdot \nabla(X\cdot \nabla |\al|^2) +  X \cdot \nabla(X\cdot \nabla
                        |\al|^2)+ (\Div N) N \cdot \nabla  |\al|^2\\\
  &\geq -C\rho(r)^2|\nabla^2 |\al|^2| -
    C\rho(r)|\nabla\nablap|\al|^2| -|\nablap\nablap|\al|^2|- C |
    \nabla |\al|^2|\\
  &\geq -C[1+r^{1/2}\rho(r) +r\, \rho(r)^2]\,.
\end{align*}

On the other hand, Proposition~\ref{P.albe} allows us to write
\begin{equation*}
 r|\al|^{4}(1-|\al|^{2}-|\be|^2)\leq
- \frac{1}{2}|\al|^{2}\Delta |\al|^{2}+ \frac{1}{2} {|\nabla
  |\al|^2|^2}
+C\big( 1+|\nablat|\al|^2| \big)\,.
\end{equation*}
If we now evaluate at~$q$ the inequalities that we have derived and
invoke the bounds obtained in Lemma~\ref{L.bounds2}, we
infer that, at the point~$q$
\begin{align*}
r|\al|^{4}(1-|\al|^{2}-|\be|^2)\leq C [1+r^{1/2} \rho(r) + r\rho(r)^2]\,.
\end{align*}
Since $\be\to0$ as $r\to\infty$ by Lemma~\ref{L.TaubesEstimates},
we finally conclude that
\[
|\al|^{4}(1-|\al|^{2})\leq   C[r^{-1}+r^{-1/2}\rho(r)+\rho(r)^2] \leq C[r^{-1/2}+\rho(r)^2]
\]
at the point~$q$. This is the bound stated in the theorem.
\end{proof}

The main strength of the maximum principle stated in Theorem~\ref{T.mprincip} is that it does not assume that the sequence of solutions $(r,\psi,A)$ has uniformly bounded energy. The following result exploits this property to show that if the energy growth is smaller than $r^{1/2}$, there are points on $M$ where $|\alpha|^2\to 0$. This turns out to be an effective alternative to Taubes's local analysis of solutions with bounded energy using the vortex equation, and it will be crucially used in the proof of Theorem~\ref{T.2}.

In the proof, it is convenient to use suitable flow boxes adapted to
the vector field~$X$. To define a flow box, let $p$ be any point in $M$ and $\{e_1,e_2,X\}$ an orthonormal basis at $T_pM$. Consider, for positive constants $\ep$ and $R$, the map
\[
\Phi_p : (0, \ep) \times \mathbb{D}_R\longrightarrow M\,,
\]
defined by
\[
\Phi_p(t, x):=\phi^{t}_{X} \big(\exp_{p} (x_1 e_1(p)+x_2 e_2 (p))\big)\,,
\]
where $\DD_R:=\{x \in \RR^2: |x| < R \}$ is the disk of
radius $R$, $\phi_X^t$ is the time-$t$ flow of~$X$, and
$\exp_p: T_p M \rightarrow M$ is the exponential map. With $\ep$
and $R$ small enough, $\Phi_p$ is a smooth diffeomorphism into its image, which we will
denote by $\cC_p(R, \ep)$. From now on, we will refer to $\cC_p(R, \ep)$ as the \emph{flow box} based at $p$ of radius $R$ and
length~$\ep$.

\begin{theorem}\label{T.pt}
Let $(r_n,\psi_n,A_n)$ be a sequence of solutions to the
Seiberg--Witten equations with $r_n\to\infty$ and such that
$\cE_n=o(r_n^{1/2})$, i.e.,
\[
\limsup_{n\to\infty}\cE_nr_n^{-1/2}=0\,.
\]
Let $\{p_n\}$ be a sequence of points in $M$ for which there is a positive constant $\te$ such that
\[
1-|\al_n|^2(p_n)\geq \te\,.
\]
Then there is a constant $L$ (independent of $n$) such that the following holds: if $n$ is large
enough, there are disks~$\Si_n\subset M$ of radius $\rho_n:= L\cE_n r_n^{-1/2}$,
transverse to the vector field~$X$
and perpendicular to it at~$p_n$, and points $q_n\in \Si_n$ such that
\[
|\al_n|^2(q_n)\leq C(r_n^{-1/4}+\rho_n)\,.
\]
\end{theorem}

\begin{proof}
The existence of the sequence of points $\{p_n\}$ in ensured by Theorem~\ref{T.existence}, so let us take disks $\Sigma_n$ centered at $p_n$ as in the statement. We claim that the bound on the energy growth ensures that there
is a local minimum of $|\al_n|^2|_{\Si_n}$ in the interior of the disk~$\Si_n$, provided that $n$ is large
enough. In order to prove this, we proceed by contradiction.

Consider the connected component $\cV_n$ of the compact set
\[
\{ q\in\overline{\Si_n}: 1-|\al_n|^2(q)\geq \te\}
\]
that contains the point~$p_n$. Let us assume that $\cV_n\cap \pd\Si_n$
is nonempty. Then there exists a continuous curve
$\Ga_n:[0,1)\to \Si_n$ with $\Ga_n(0)=p_n$ and
$\overline{\Ga_n([0,1))}\cap \pd\Si_n\neq \emptyset$ such that
\[
1-|\al_n|^2(\Ga_n(s))\geq \te
\]
for all $s\in [0,1)$.

Take a small enough constant $\La>0$ that will be fixed later. Since the length of the curve $\Ga_n([0,1))$ is
at least $\rho_n$, one can take at least
$$K_n:=\frac{r_n^{1/2}\rho_n}{2 \La\theta}$$
pairwise disjoint flow-boxes
\[
\cC_{p_{n,k}}\bigg(\frac{\La\te}{\srn}, \La\te
\bigg)\,
\]
centered at different points $\{p_{n, k}\}_{k=1}^{K_n}$ lying on the image of the curve $\Gamma_{n}$, with $p_{n,1}:= p_n$.
If $\La<\La_0$, Lemma~\ref{L.charge1} below and the
definition of~$\rho_n$ imply that the signed measures $\si_n$ (cf. Equation~\ref{sigma_n}) satisfy
\begin{align*}
\si_n(M)\geq \bigcup_{k=1}^{K_n}\si_n\bigg(\cC_{p_{n,k}}\bigg(\frac{\La\te}{\srn}, \La\te
\bigg)\bigg)-C\cE_n^{-1} &\geq \frac{c_0}{2} \La^2 \te^3 r_n^{1/2}\rho_n \cE_n^{-1}-C\cE_n^{-1}\\
&\geq \frac{c_0}{2} \La^2 \theta^{3} L-C\cE_n^{-1}\,.
\end{align*}
Here the constant $c_0$ comes from Lemma~\ref{L.charge1} and does not depend on $n$ or $L$.

We then infer that picking a large enough constant~$L$ in the definition of $\rho_n$ yields a contradiction with the fact that $\sigma_{n}(M)=1$ (even in the case that $\cE_n$ is uniformly bounded). Therefore, $\cV_n \cap \partial \Sigma_n$ is empty and the compactness of $\cV_n$ implies that, for large enough~$n$, there is a global minimum~$q_n$ of~$|\al_n|^2|_{\cV_n}$ on~$\cV_n$.

Since $|\al_n|^2(q_n)\leq 1-\te$, the maximum principle stated in
Theorem~\ref{T.mprincip} allows us to write the bound
\[
|\al_n|^2(q_n)\leq C(r_n^{-1/4}+\rho_n)\,,
\]
which completes the proof of the theorem.
\end{proof}

The following technical lemma is invoked in the proof of Theorem~\ref{T.pt}. We use the same notation as before.

\begin{lemma}\label{L.charge1}
 Let $(r,\psi,A)$ be a sequence of solutions to the modified Seiberg--Witten
 equations. Assume that $p$ is a point in $M$ such that
 $1-|\al|^2(p)\geq \te$ for some uniform $0<\theta<1$. Then there are positive constants~$c_0$ and~$\La_0$,
 independent of~$\te$ and~$r$, such that
\[
\sigma_r \bigg(\cC_{p}\bigg(\frac{\La\te}{\sr}, \La\te
\bigg)\bigg) \geq \frac{ c_0\La^3\te^4}{\cE_r}
\]
for all $\La<\La_0$.
\end{lemma}

\begin{proof}

First, Equation~\eqref{sigma_n} implies that, for any open set $U \subset M$

\[
 \sigma_r (U) \cE_r \geq r \int_U (1-|\al|^2)\mu-C \mu(U)
\]
for some constant $C$ independent of $r$.

Since $1 -|\al|^2\geq\te$ at $p$, it follows from the a priori estimates for the derivatives of~$|\al|^2$
in Lemma~\ref{L.bounds2} that
\[
1 -|\al|^2\geq\frac\te2
\]
in a flow box of the form $\cC_{p}(\frac{\La\te}{\sr}, \La\te)$, provided
that the constant $\La$ is smaller than some constant~$\La_0$ (independent of $r$ and $\theta$). Therefore
\begin{align*}
\cE_r \sigma_r \bigg(\cC_{p}\bigg(\frac{\La\te}{\sr}, \La\te
\bigg)\bigg) \geq r
\int_{\cC_{p}(\frac{\La\te}{\sr}, \La\te
)} (1 -|\al|^2)\mu-C\mu\bigg(\cC_{p}\bigg(\frac{\La\te}{\sr}, \La\te
\bigg)\bigg) \geq
                                                           \\
\geq \mu\bigg(\cC_{p}\bigg(\frac{\La\te}{\sr}, \La\te
\bigg)\bigg) \bigg(\frac{\te r}2-C \bigg)\geq
 C\La^3\te^4+O(r^{-\frac{1}{2}})
\geq c_0\La^3\te^4\,,
\end{align*}
as claimed.

\end{proof}

\section{Nodal sets and limiting measures: proof of Theorem~\ref{T.2}}\label{S.T2}

In most of this section we are concerned with solutions of the modified Seiberg--Witten equations whose energy is bounded as
\[
\limsup_{n\to\infty}\frac{\cE_n}{r_n^{1/2}}=0\,.
\]
Specifically, in Sections~\ref{S.step1} and~\ref{S.step2} we prove Theorem~\ref{T.2}, which establishes a connection between some invariant sets of the vector field $X$ and the set of points where $|\alpha_n|\to 0$. Our proof exploits the new maximum principle presented in Theorem~\ref{T.pt}. In particular, since it applies to solutions with uniformly bounded energy, this allows us to obtain an alternative proof of Taubes's theorem on the existence of periodic orbits without using the vortex equations, as discussed in Section~\ref{S.orbits}. Finally, in Section~\ref{S.open1}, we prove Proposition~\ref{L.con1}, which is a sort of converse to Theorem~\ref{T.2}: the open sets of $M$ where $|\alpha_n|\to 1$ do not charge the invariant measure $\sigma_\infty$. No constraint on the energy growth is assumed in this case.

\subsection{Step 1: construction of an invariant set}\label{S.step1}

We first observe that we can define the sets $Z_n^\theta$ and $Z_n$ using $|\alpha_n|^2$ rather than $|\psi_n|^2$ (by the a priori estimates in Lemma~\ref{L.TaubesEstimates}). Let us pick any $\theta\in(0, 1)$. By Theorem~\ref{T.existence}, the compact set
\[
Z^{\theta}_n:=\{p \in M,\, 1-|\al_n|^2(p)\geq\te \}
\]
is non-empty for $\theta$ small enough, and in fact, by Theorem \ref{T.mprincip}, it is non-empty for any $\te\in(0, 1)$ and all large enough $n$.

Fix a subsequence which converges in the Hausdorff metric, which we still denote by $Z_n^\theta$, and let $Z^{\theta}_{\infty}$ be its limit. In this step, our aim is to show that $Z^{\theta}_{\infty}$ is invariant under the flow of $X$. Notice that by compactness of $M$, non-empty compact subsets of $M$ with the Hausdorff metric form a compact metric space, so such a limiting set always exists and it is not the empty set.

Before proceeding, let us explain the main idea of the proof. For any point $p \in Z^{\theta}_{\infty}$, we will show that there is a set $S_p$ containing $p$, contained in $Z^{\theta}_{\infty}$, and invariant under the flow of $X$. This clearly implies that $Z^{\theta}_{\infty}$ is invariant. The set $S_p$ will turn out to be the closure of the orbit of $X$ passing through $p$.

The Hausdorff convergence implies that, for any $p \in Z^{\theta}_{\infty}$, there is a sequence of points $p_n \in Z^{\theta}_n$ converging to $p$. By definition, the points $p_n$ satisfy
\[
1-|\al_n|^2(p_n)\geq\te
\]
for all~$n$, with $\te>0$.

Since $\cE_n=o(r_n^{-1/2})$, it follows from Theorem~\ref{T.pt} that for each large enough $n$, there exists a point $q_n^1$ on a disk $\Si_n^1$
  centered at~$p_n$, of radius at most
  \[
    \rho_n:=cr_n^{-1/2}\cE_n
  \]
  and orthogonal to~$X$
  at~$p_n$, such that
  \[
|\al_n(q_n^1)|^2<\ep_n\,,
\]
where $\ep_n:=C(r_n^{-1/4}+ \rho_n)$.  We then consider the cylinder
\[
  \cC_{n,1}:=\cC_{q_n^1}(\rho_n,T)
\]
of radius~$\rho_n$ centered at this point and length $T$, where~$T$ is a small positive constant
independent of~$n$.

Consider the point $\tq_n^1:=\phi_X^T(q_n^1)$ and take a
  disk $\Si_n^2$ centered at~$\tq_n^1$, of radius $\rho_n$ and orthogonal to~$X$
  at~$\tq_n^1$. The bounds for the derivatives of~$|\al|^2$
  (Lemmas~\ref{L.TaubesEstimates} and~\ref{L.bounds2}) ensure that
  \[
|\al_n(\tq_n^1)|^2<\ep_n+ CT\,.
  \]
Hence if $T$ is small, Theorem~\ref{T.pt} ensures again that there exists a point
$q_n^2\in\Si_n^2$ such that
\[
|\al_n(q_n^2)|^2<\ep_n\,.
\]
Let us now define the flow-box
\[
\cC_{n,2}:=\cC_{q_n^2}(\rho_n,T)
\]
and note that the volumes of~$\cC_{n,k}$ (with $k=1,2$) and of the
intersection $\cC_{n,1}\cap \cC_{n,2}$ can be estimated as
\begin{align*}
  \mu(\cC_{n,k}) &> C\rho_n^2\,,\\
  \mu(\cC_{n,1}\cap \cC_{n,2}) &< C\rho_n^3\,.
\end{align*}
Since $\rho_n\to0$, this means that, for large enough~$n$,  the volume of the
intersection~$\cC_{n,1}\cap \cC_{n,2}$ is just a small fraction of that of
either of the cylinders.

By repeating the argument (considering both the forward flow of~$X$
and the backward flow), one obtains a sequence of points $(q_n^k)_{k\in\ZZ}$,
which give rise to flow boxes $\cC_{n,k}:=\cC_{q_n^k}(\rho_n,T)$
satisfying
\begin{align*}
  \mu(\cC_{n,k}) &> C\rho_n^2\,,\\
  \mu(\cC_{n,k}\cap \cC_{n,k+1}) &< C\rho_n^3\,,\\
    |\al_n(q_n^k)|^2&<\ep_n
\end{align*}
for all~$k$ and all large enough $n$.

Consider now, for some constant $D$ independent of $n$, the thinner cylinders $\widetilde{\cC_{n, k}}:=\cC_{q_n^k}(Dr_n^{-\frac{1}{2}},T)\subset \cC_{n,k}$. If $D$ and $T$ are chosen small enough, the bounds for the derivatives of~$|\al|^2$
in Lemmas~\ref{L.TaubesEstimates} and~\ref{L.bounds2} ensure that for any point $q' \in \widetilde{\cC_{n, k}}$ we have
\begin{equation}\label{thincyl}
|\al_n(q')|^2 <\ep_n+C (D+T) < 1-\te
\end{equation}
for all $k$ and all large enough $n$. For each positive integer~$K$, let us set
\[
S_{n,K}:= \bigcup_{k=-K}^K \widetilde{\cC_{n,k}}\,.
\]
By construction, $S_{n,K}$ is contained in a neighborhood of width
$K\rho_n$ of the portion
\[
S_{K}^n:= \{\phi_X^t p_n : |t|\leq KT\}
\]
of the integral curve of~$X$ passing
through~$p_n$. If the integral curve is periodic, this
length may mean that this set winds around the integral curve more than
once. In particular, setting
\[
S_{K}:= \{\phi_X^t p : |t|\leq KT\}\,,
\]
it is clear that both $S_{n,K}$ and $S_{K}^n$ converge to $S_K$ as $n\to\infty$, albeit this convergence does not need to
be uniform in~$K$.

Finally, let us define the compact invariant
set~$S_p$ as the closure of the integral curve of~$X$ passing
through~$p$, which obviously arises as the Hausdorff limit of~$S_K$ as $K\to\infty$. We claim that $S_p \subset Z^{\te}_{\infty}$.

To see this, observe that Equation \eqref{thincyl} ensures that, for all $K$ and any large enough $n$, $S_{n,K} \subset Z^{\te}_n$. Consider an infinite sequence of integers $\cdots < K_{-1} < K_0 < K_1 <\cdots$. It is clear that any $q \in S_p$ is the limit as $|i| \rightarrow \infty$ of some sequence of points $q_i \in S_{K_i}$, and the points $q_i$ are themselves the limits as $n\rightarrow \infty$ of some sequence of points $p^{n}_i \subset S_{n,K_i} \subset Z^{\te}_n$. Upon choosing a diagonal sequence of $p^{n}_i$, we conclude that $q$ is the limit as $(n, |i|) \rightarrow (\infty, \, \infty)$ of a sequence of points in $ Z^{\te}_n$. The uniqueness of the Hausdorff limit (recall that we have fixed a converging subsequence at the beginning) implies that $q \in  Z^{\te}_{\infty}$, as we wanted to prove.

\subsection{Step 2: the collection of limiting sets is independent of $\theta$}\label{S.step2}

Let us recall the definition of the sets $Z_n$:
\[
Z_n:=\left\{p \in M: |\al_n|^2(p) \leq C\text{ max }(r_n^{-\frac{1}{4}},\cE_n r_n^{-\frac{1}{2}}) \right\}\,.
\]
Notice that $Z_n \subset Z^{\theta}_n$ for all $n$ large enough.

We claim that given any $\te \in (0, 1)$, and any converging subsequence $Z^{\te}_{n}$ (in the Hausdorff metric), there is a converging subsequence $Z_{n}$ such that the limits coincide: $Z^{\te}_{\infty}=Z_{\infty}$. Reciprocally, given a convergent subsequence $Z_n$, there is a subsequence $Z^{\te}_{n}$ with the same limit.

Recall that the Hausdorff distance between the sets $Z_{n}$ and $Z^{\theta}_{n}$ is defined as
\[
\dist _{H}(Z_{n}, Z^{\theta}_{n})=\text{max}\bigg(\sup_{x \in Z_n} \dist (x, Z^{\theta}_{n}), \sup_{y \in Z^{\theta}_{n}} \dist (y, Z_{n})\bigg)\,,
\]
for each $n$. Fix a converging subsequence $Z^{\te}_{n}$, and consider the corresponding sequence of sets $Z_n$. We claim that $\dist _{H}(Z_{n}, Z^{\theta}_{n}) \rightarrow 0$ as $n\rightarrow \infty$.

Indeed, by Theorem~\ref{T.pt}, for any sequence of points $p_n \in Z^{\theta}_n$ we can find another sequence $q_n$ such that
\[
|\al_n|^2(q_n)\leq C(r_n^{-1/4}+\rho_n)\,,
\]
\[
\dist (p_n, q_n) < \rho_n \,,
\]
with $\rho_n:=c\cE_n r^{-\frac12}_{n}$ going to zero as $n\rightarrow \infty$. From this we infer that $q_n \in Z_n$ for all $n$ and we conclude that
\[
\sup_{y \in Z^{\theta}_{n}} \dist (y, Z_{n}) < \rho_n \rightarrow 0
\]
as $n\to\infty$. Since $Z_n \subset Z^{\theta}_n$ for all large enough $n$, we can write
\[
\sup_{x \in Z_n} \dist (x, Z^{\theta}_{n})=0\,,
\]
thus implying that $\dist _{H}(Z_{n}, Z^{\theta}_{n}) \rightarrow 0$ as claimed. In particular, $Z_n$ converges to a compact set $Z_\infty$ which is equal to the limiting set $Z^\theta_\infty$. The same argument shows that if a subsequence $Z_n$ converges to a set $Z_\infty$, the corresponding subsequence $Z_n^\theta$ converges to the same Hausdorff limit. Analogously, combining the previous argument with the fact that $Z_n^{\theta'}\subset Z_n^\theta$ if $\theta'\geq \theta$, it follows that any converging subsequence $Z_n^\theta$ yields a converging subsequence $Z_n^{\theta'}$ with the same limit, for any $\theta'\neq\theta$. This completes the proof of Theorem~\ref{T.2}.

%
%
%
%
%
%

\subsection{The bounded energy case: Taubes's result revisited}\label{S.orbits}

All the previous arguments, as well as the maximum principle proved in Section~\ref{S.max}, apply to sequences of solutions with uniformly bounded energy. In fact, in this case, a simple boundedness argument allows us to prove that the invariant set $Z_{\infty}$ must consist of a finite collection of periodic orbits of $X$. Of course, this recovers Taubes's periodic orbit theorem, but without making use of the local analysis that compares Seiberg--Witten with the vortex equations. To see this, in this short subsection we shall assume that $\cE_n \leq C$.

We claim that for any $p \in Z^{\te}_{\infty} $, the invariant set $S_p$ that we constructed in Section~\ref{S.step1} is a periodic orbit of $X$. Indeed, suppose $S_{p}$ is not a periodic orbit. Then one has that, for any $K$ as large as desired, the cylinders $\widetilde{\cC_{n,k}}$ satisfy the
small intersection condition
\[
\mu( \widetilde{\cC_{n,j}}\cap \widetilde{\cC_{n,k}})<C \overline{\rho_n}^3
\]
for all $-K\leq j< k\leq K$ and all large enough~$n$ (depending
on~$K$). Here, we set $\overline{\rho_n}:=D r^{-\frac{1}{2}}$.

Let us now define the slightly cut out cylinders
\[
\cC_{n,k}':= \widetilde{\cC_{n,k}}\backslash\overline{\widetilde{\cC_{n,k-1}}}\,,
\]
which are pairwise disjoint by construction. In view of Equation~\eqref {thincyl} and Lemma~\ref{L.charge1}, we have
\[
\cE_n \si_n(\cC_{n,k}')>\delta
\]
for some constant $\delta>0$ depending on $\theta, T$ and $D$ (which are taken sufficiently small), but not on $n$.

Observe that, by definition, $\cE_n \si_n(M)=\cE_n$. Moreover, the sets $\cC_{n,k}'$ are pairwise disjoint, so Equation~\eqref{sigma_n} and the bound for the negative part of $1-|\alpha|^2$ in Lemma~\ref{L.TaubesEstimates} imply
\[
\cE_n\geq \sum_{k=-K}^K \cE_n \si_n(\cC_{n,k}')-C\geq (2K+1) \delta-C\,,
\]
where all the constants are independent of $n$. Since $K$ can be taken as large as desired and $\cE_n$ is uniformly bounded by hypothesis, this yields a contradiction. So~$S_p$ must be a periodic orbit.

The same argument shows that the number of periodic orbits in $Z^{\te}_{\infty}$ must be finite; otherwise we could construct an unbounded number of disjoint cylinders with $\cE_n\sigma_n$ bounded from below, contradicting the boundedness of the energy.

\subsection{Additional concentration properties: proof of Proposition~\ref{L.con1}}\label{S.open1}

In this final section we prove Proposition~\ref{L.con1}. This concerns the set of points where $|\alpha_n|\to 1$, which is not considered in Theorem~\ref{T.2}. We show that any open component of such a set has zero measure with respect to $\sigma_\infty$. The only hypothesis on the energy sequence is that it is unbounded. We stress that for sequences of bounded energy, an analogous result follows from the analysis in~\cite{T07}, which makes use of the convergence of the Seiberg--Witten equations towards the vortex equations at small scales.

\begin{proof}[Proof of Proposition~\ref{L.con1}]
We first observe that the assumption $|\psi_n|\to 1$ on $U$ is equivalent to $|\alpha_n|\to 1$ on $U$, by the a priori estimates (Lemma~\ref{L.TaubesEstimates}).

Let us define $v_n:=|\al_n|^2-1$. Equation~\eqref{eqal2} in Proposition~\ref{P.albe} can be written as:
\begin{equation}\label{eqv}
(1+v_n)\Delta v_n-  |\nabla v_n|^2=2r(1+v_n)^2 v_n+H(v_n)\,,
\end{equation}
where the term $H(v_n)$ satisfies the pointwise bound
\begin{equation}\label{hbound}
|H(v_n)| \leq C\Big(1+|\nablat v_n|\Big)\,.
\end{equation}
Here we have used that $|\beta|^2\leq \frac{C}{r}$, cf. Lemma~\ref{L.TaubesEstimates}.

Fix a sufficiently small constant $\rho>0$, and consider a geodesic ball $B_{\rho}\subset U$ of radius $\rho$. It is convenient to define a cut-off $C^\infty$ function $\chi: M\rightarrow [0,1]$ that vanishes on the complement of $B_{\rho}$, is positive on $B_\rho$ and is equal to $1$ on $B_{\frac{\rho}{2}}\subset B_{\rho}$. It is easy to see that $\chi$ can be chosen to satisfy the pointwise bounds:
\begin{equation}\label{chi1}
|\nabla \chi|^2 \leq \frac{C}{\rho^2}\,,\qquad  |\Delta \chi| \leq \frac{C}{\rho^2}\,.
\end{equation}

Now, multiply Equation~\eqref{eqv} by $\chi^2 v_n$ to obtain:
\[
\chi^2 v_n(1+v_n)\Delta v_n-  \chi^2 v_n|\nabla v_n|^2=2r\chi^2(1+v_n)^2 v^2_n+\chi^2 v_n H(v_n)\,.
\]
Taking into account that
\[
\Delta(\chi^2 v_n)=v_n \Delta \chi^2+\chi^2 \Delta v_n+4\chi \nabla \chi \cdot \nabla v_n\,,
\]
we can write
\begin{align*}
v_n(1+v_n)\Delta(\chi^2 v_n)-v^2_n (1+v_n) \Delta \chi^2 - 4 v_n(1+v_n)\chi\nabla \chi \cdot \nabla v_n-\chi^2 v_n |\nabla v_n|^2\\=2r\chi^2(1+v_n)^2 v^2_n+\chi^2 v_n H(v_n)\,.
\end{align*}
If we integrate this equation over the ball $B_{\rho}$, and integrate by parts the term $v_n(1+v_n)\Delta(\chi^2 v_n)$, we get the following expression for the local $L^2$ norm of the derivatives of $v_n$:
\begin{align*}\label{inteq}
-\int_{B_{\rho}} \chi^2|\nabla v_n|^2=3 \int_{B_{\rho}} v_n \chi^2|\nabla v_n|^2+8 \int_{B_{\rho}}v^2_n\chi\nabla \chi \cdot \nabla v_n\\+6 \int_{B_{\rho}}v_n\chi\nabla \chi \cdot \nabla v_n+ \int_{B_{\rho}}v^2_n (1+v_n) \Delta \chi^2\\+2r \int_{B_{\rho}} \chi^2 v^2_n(1+v_n)^2+\int_{B_{\rho}} \chi^2 v_n H(v_n)\,.
 \end{align*}
Our objective now is to bound the integrals in the right hand side of this equation. We will use repeatedly the bounds in Equation~\eqref{chi1} and the fact that,  for any $\delta>0$,
\[
\int_{B_{\rho}} \chi|\nabla v_n| \leq \delta \int_{B_{\rho}}\chi^2 |\nabla v_n|^2+\frac{C}{\delta}\,,
\]
with $C$ a constant independent of $\delta$ and $n$. This follows from the elementary inequality
\[
\chi f \leq \delta \chi^2 f^2+\frac{1}{4 \delta}\,.
\]

Noting that the $L^\infty$ norm $\|v_n\|_{\infty}$ of $v_n$ on the ball $B_\rho$ is bounded by $1$ (as a consequence of the a priori bound on $|\alpha_n|^2$), we have
\[
\Big| \int_{B_{\rho}} v_n \chi^2|\nabla v_n|^2 \Big| \leq \|v_n\|_\infty\int_{B_{\rho}} \chi^2|\nabla v_n|^2\,,
\]
\[
\Big|\int_{B_{\rho}}v^2_n\chi \nabla \chi \cdot \nabla v_n \Big| \leq  \frac{C\|v_n\|^2_\infty}{\rho} \int_{B_{\rho}} \chi |\nabla v_n| \leq \frac{C\|v_n\|_\infty}{\rho} \Big( \delta \int_{B_{\rho}}\chi^2 |\nabla v_n|^2+\frac{C}{\delta}\Big)\,,
\]

\[
\Big| \int_{B_{\rho}} v_n\chi\nabla \chi \cdot \nabla v_n \Big| \leq \frac{C\|v_n\|_\infty}{\rho} \int_{B_{\rho}}\chi |\nabla v_n| \leq  \frac{C\|v_n\|_\infty}{\rho} \Big( \delta \int_{B_{\rho}}\chi^2 |\nabla v_n|^2+\frac{C}{\delta}\Big)\,,
\]
\[
\Big| \int_{B_{\rho}}v^2_n (1+v_n) \Delta \chi^2 \Big| \leq \frac{C\|v_n\|_\infty}{\rho^2}\,.
\]
Furthermore, using the bounds for $H$ in Equation~\eqref{hbound}, we deduce
\[
\Big| \int_{B_{\rho}} v_n \chi^2 H \Big| \leq C\|v_n\|_\infty \Big(\rho^3+\int_{B_{\rho}} \chi|\nabla v_n|\Big) \leq C\|v_n\|_\infty\Big(\rho^3+\delta \int_{B_{\rho}}\chi^2 |\nabla v_n|^2+\frac{C}{\delta}\Big)\,.
\]
Finally, applying Equation~\eqref{sigma_n}, we obtain
\[
2r \int_{B_{\rho}} \chi^2 v^2_n(1+v_n)^2 \leq C\|v_n\|_\infty \Big(\cE_n  |\sigma_n(B_{\rho})|+1\Big)\,.
\]
Plugging all these estimates into the integral identity we obtained for $\chi^2|\nabla v_n|^2$, and using that $\delta$ and $\rho$ are small, we conclude:
\[
\Bigg(1-\|v_n\|_{\infty}\Big(3+\frac{C\delta}{\rho}\Big)\Bigg)\int_{B_{\rho}}\chi^2 |\nabla v_n|^2 \leq C\|v_n\|_{\infty} \Big(\frac{1}{\rho^2 \delta}+\cE_n |\sigma_n(B_{\rho})|\Big)
\]
Now we use that, by assumption, $\|v_n\|_{\infty}$ goes to zero as $n\rightarrow \infty$ (because $|\alpha_n|\to 1$ on $U\supset B_{\rho}$). Fixing a constant $\delta$, for large enough $n$ we have
\[
\int_{B_{\frac{\rho}{2}}} |\nabla v_n|^2\leq \int_{B_{\rho}}\chi^2 |\nabla v_n|^2 \leq C\|v_n\|_{\infty} \Big(\frac{1}{\rho^2 \delta}+\cE_n |\sigma_n(B_{\rho})|\Big)\,.
\]
Therefore, as $n \rightarrow \infty$
\begin{equation}\label{gradientcero}
\frac{1}{\cE_n}\int_{B_{\frac{\rho}{2}}} |\nabla |\alpha_n|^2|^2=\frac{1}{\cE_n}\int_{B_{\frac{\rho}{2}}} |\nabla v_n|^2\leq  C\|v_n\|_{\infty} \Big(\frac{1}{\rho^2 \delta\cE_n}+|\sigma_n(B_{\rho})|\Big)\rightarrow 0\,,
\end{equation}
which holds even for solutions with uniformly bounded energy $\cE_n$.

To show that $\sigma_{\infty}(U)=0$, we first observe that, for $n$ large enough we have
\[
\frac{1}{\sqrt{2}}\leq|\al_n|^2\leq 1+Cr^{-1}_{n}
\]
at any point on $B_\rho$, the upper bound coming from Lemma~\ref{L.TaubesEstimates}. Together with Equation~\eqref{sigma_n}, this implies that
\begin{align}\label{eq.otra}
\frac{r_n}{\cE_n}\int_{B_{\frac{\rho}{4}}} |\al_n|^4 (1-|\al_n|^2)  \geq \frac{r_n}{2 \cE_n} \int_{B_{\frac{\rho}{4}}} (1-|\al_n|^2)-\frac{C}{\cE_{n}} \rho^3 \geq \frac{1}{2} \sigma_n(B_{\frac{\rho}{4}})-\frac{C}{\cE_{n}} \rho^3\,.
\end{align}
Now, let $\chi':M\to [0,1]$ be a smooth cut-off function supported on the ball $B_{\frac{\rho}{2}}$, equal to one on $B_{\frac{\rho}{4}}$ and positive on $B_{\frac{\rho}{2}}$. We assume that it satisfies the same bounds as in~\eqref{chi1}. Multiplying Equation~\eqref{eqal2} by $\chi'$ and integrating, we deduce
\[
\frac{2r_n}{\cE_n}\int_{B_{\frac{\rho}{4}}} |\al_n|^4 (1-|\al_n|^2)\leq \frac{1}{\cE_n}\Bigg( \int_{B_{\frac{\rho}{2}}} \chi' |\nabla |\al_n|^2|^2-\int_{B_{\frac{\rho}{2}}} \chi' |\al_n|^2 \Delta |\al_n|^2+ \int_{B_{\frac{\rho}{2}}} \chi' H +C\rho^3\Bigg)\,,
\]
where we have used the a priori bounds for $|\beta_n|^2$ and that $1-|\alpha_n|^2+Cr_n^{-1}\geq 0$. Accordingly, from Equation~\eqref{eq.otra} we get
\begin{equation}\label{sbound}
\sigma_n(B_{\frac{\rho}{4}}) \leq \frac{1}{\cE_n}\Bigg( \int_{B_{\frac{\rho}{2}}} \chi' |\nabla |\al_n|^2|^2-\int_{B_{\frac{\rho}{2}}} \chi' |\al_n|^2 \Delta |\al_n|^2+ \int_{B_{\frac{\rho}{2}}} \chi' H +C\rho^3\Bigg)\,.
\end{equation}
Next, let us estimate the second term on the right hand side of this equation. Integrating by parts we obtain
\[
-\int_{B_{\frac{\rho}{2}}} \chi' |\al_n|^2 \Delta |\al_n|^2=\int_{B_{\frac{\rho}{2}}} \chi' |\nabla |\al_n|^2|+\int_{B_{\frac{\rho}{2}}}  |\al_n|^2 \nabla \chi' \cdot \nabla |\al_n|^2\,.
\]
By the elementary inequality $a \leq a^2+\frac{1}{4}$, the bound for $|\nabla\chi'|$ and the fact that $|\al_n|^2 \leq 1+Cr_n^{-1}$, we can write
\[
\int_{B_{\frac{\rho}{2}}}  |\al_n|^2 \nabla \chi' \cdot \nabla |\al_n|^2 \leq \frac{C}{\rho} \int_{B_{\frac{\rho}{2}}}  |\nabla |\al_n|^2|^2+C\rho^2\,,
\]
with $C$ independent of $\rho$ and $n$. Summing up, we obtain the bound
\[
-\int_{B_{\frac{\rho}{2}}} \chi' |\al_n|^2 \Delta |\al_n|^2 \leq \frac{C}{\rho}\int_{B_{\frac{\rho}{2}}}  |\nabla |\al_n|^2|^2+C\rho^2\,.
\]

Using this estimate in Equation~\eqref{sbound}, it follows that
\[
\sigma_n(B_{\frac{\rho}{4}})\leq \frac{1}{\cE_n}\bigg( \frac{C}{\rho}\int_{B_{\frac{\rho}{2}}} |\nabla |\al_n|^2|^2+\int_{B_{\frac{\rho}{2}}} \chi' H +C\rho^2\bigg)\leq \frac{C}{\cE_n}\bigg( \frac{1}{\rho}\int_{B_{\frac{\rho}{2}}} |\nabla |\al_n|^2|^2+\rho^2\bigg)\,,
\]
where we have used that Equation~\eqref{hbound} implies the bound
\[
\bigg|\int_{B_{\frac{\rho}{2}}}\chi' H\bigg| \leq C\int_{B_{\frac{\rho}{2}}} |\nabla |\al_n|^2|^2+C\rho^3\,.
\]
Then we infer from Equation~\eqref{gradientcero} and the assumption $\cE_n \rightarrow \infty$, that
\begin{equation}\label{eq.zeromass}
\sigma_{n}(B_{\frac{\rho}{4}}) \rightarrow 0
\end{equation}
as $n\to\infty$. Since for any point $p\in U$ we can take a small enough neighborhood $N_p$ whose closure is contained in $U$, Equation~\eqref{eq.zeromass} implies that $\sigma_\infty(N_p)=0$, thus completing the proof of the proposition.
\end{proof}

\section{Absence of local obstructions for the invariant measures}
\label{S.3}

In this section we prove Theorem~\ref{T.3}. To this end, in Section~\ref{SS.vortex} we show that, locally, the modified Seiberg--Witten equations can be reduced to a rescaled version of the vortex equations. This allows us to study the limiting invariant measures using the 2-dimensional vortex equations, cf. Section~\ref{SS.teo}.

As defined before stating Theorem~\ref{T.pt}, we denote by $\cC$ a flow box adapted to the vector field $X$. We recall that a flow box is the image of the cylinder $(0, 1) \times \mathbb{D}$ under an appropriate map
\[
\Phi : (0, 1) \times \mathbb{D}\longrightarrow M\,,
\]
which is a diffeomorphism into its image and which satisfies
\[
d \Phi (\partial_t)=X\,.
\]
Here $t$ is the coordinate in the interval $(0, 1)$. By the volume-preserving flow box theorem, we can choose the local diffeomorphism $\Phi$ so that the volume form $\mu$ on the flow box coordinates is given by
\[
\mu=C dx \wedge dy \wedge dt
\]
for some small enough constant $C$, and coordinates $(x, y) \in \DD$, $t \in (0, 1)$.

The standard Euclidean metric
\[
g_{0}=dx^2+dy^2+dt^2
\]
is then an adapted metric for the vector field $X$ on $\cC$. It is easy to see that we can construct a global metric $g$ on $M$ adapted to the vector field $X$ so that $\Phi_{*} g_0=g|_{\cC}$.

\subsection{From Seiberg--Witten to the rescaled vortex equations}\label{SS.vortex}

In this section we use the notation and constructions introduced in Section~\ref{S.defs}. We always work in the flow box $\cC$ using the aforementioned coordinates and adapted metric. The $1$-form $\lambda=i_{X} g$ is $dt$, the Hermitian line bundle $K=\Ker\lambda$ is spanned by the vector fields $\{\partial_{x}, \partial_{y}\}$ and thus it is trivial, and the base connection $A_0$ defined by Equation~\eqref{fiducial} is $A_0=0$, and then the 1-form $\bom_{K}$ introduced in Equation~\eqref{bom} is also $0$.

We take the associated line bundle $E$ to be the trivial bundle $\CC \times \cC$. We endow the rank-two complex bundle $\SS=E\oplus K^{-1}E$ with the following spin structure, defined via the Clifford multiplication:
\[
\sigma(\partial_t):=\begin{pmatrix}  i & 0 \\ 0 & -i \end{pmatrix}, \qquad
\sigma(\partial_y):=\begin{pmatrix}  0 & -1 \\ 1 & 0 \end{pmatrix},  \qquad
\sigma(\partial_x):=\begin{pmatrix}  0 & i \\ i & 0 \end{pmatrix}\,.
\]
Since all the bundles are trivial, the spinor can be identified with a map $\psi=(\alpha, \beta) : \cC \rightarrow \CC^2$ and the connection with a $1$-form $A=A_t dt+A_x dx+ A_{y} dy$ on $\cC$. The modified Seiberg--Witten equations then read as
\begin{align}\label{SWbox}
\partial_y A_{x}-\partial_x A_{y}=r(1-|\alpha|^2+|\beta|^2)\,,
\\
\partial_x A_{t}-\partial_t A_x=ir(\bar{\alpha}\beta-\bar{\beta} \alpha)\,,
\\
\partial_t A_y-\partial_y A_t=r(\bar{\alpha}\beta+\bar{\beta}\alpha)\,,
\end{align}
and the second equation (the Dirac equation) is
\begin{align}\label{SWbox2}
-(\partial_t \beta -iA_t \beta)+(\partial_x-i\partial_y) \alpha-i(A_x-iA_y)\alpha=0\,,\\
(\partial_t \alpha +iA_t \alpha)+(\partial_x+i\partial_y) \beta+i(A_x+iA_y)\beta=0\,.
\end{align}

These equations can be simplified if we look for $t$-independent solutions that satisfy
\begin{equation}\label{ans}
\beta= A_t= \partial_t A_x=\partial_t A_y=\partial_t \alpha=0\,,
\end{equation}
in which case the modified Seiberg--Witten equations reduce to the well-known rescaled vortex equations on $\DD \subset \CC$ with the complex variable $z:=x+i y$:
\begin{align}\label{eq.vort}
\star d a=r(1-|\phi|^{2})\,,
\\
\overline{\partial}_{a}\phi:=\overline{\partial}_z\phi-i(a_{x}-ia_{y})\phi=0\,.\label{eq.vort2}
\end{align}
Here we have set $\phi:=\alpha$ and $a=a_x dx+a_y dy:=A_x dx+A_y dy$. Equations~\eqref{eq.vort} and~\eqref{eq.vort2} are obtained from the standard vortex equations using the change of variables $z=\sr z'$ (see e.g.~\cite{JT}).

The finite-energy solutions to the vortex equations are well understood. In particular, the following result was proved by Taubes, see~\cite{Taubes80,JT}. It will be instrumental to prove Theorem~\ref{T.3}, so we state it for future reference.

\begin{theorem}[Taubes~\cite{Taubes80, JT}]\label{vortextaubes80}
Let $\cP:=\{z_j\}_{j=1}^k$ be a finite set of distinct points $z_j\in\Cone$, and let $\{m_j\}_{j=1}^k$ be an associated set of positive integers. There is a smooth solution $(a, \phi)$ to the vortex equations~\eqref{eq.vort} and~\eqref{eq.vort2} with $r=1$ such that $\phi^{-1}(0)=\cP$, and such that the zero $z_j$ of $\phi$ has multiplicity $m_j$. Furthermore, the solution satisfies the additional properties:
\begin{enumerate}
\item $|\phi| < 1$ on $\CC$ and $|\phi|\rightarrow 1$ as $|z|\rightarrow \infty$.

\item The energy of the solution is given by
\[
\cE:=\int_{\CC} da=\int_{\CC} (1-|\phi|^2)=2\pi \sum_j m_j\,.
\]

\item There is a universal constant $C$, not depending on the particular configuration of points $\cP$ nor on their multiplicities, such that
\[
|\nabla |\phi|^{2}(z)|\leq |\nabla_a \phi (z)| \leq C\,.
\]

\item Let $\Omega^{-}(\phi)$ denote the set of points in $\Cone$ where $|\phi|^{2}\leq \frac{1}{2}$. There is a universal constant $c\in (0, 1)$, not depending on the particular configuration of points nor their multiplicities, such that, for any $z\in \Cone$ with $\dist (z, \Omega^{-}(\phi))\geq c^{-1} $,
\begin{equation}\label{vortexdecay1}
|1-|\phi(z)|^{2}|\leq e^{-c\dist (z,  \Omega^{-}(\phi))}\,,
\end{equation}
\begin{equation}\label{vortexdecay2}
|\nabla |\phi|^{2}(z)|\leq |\nabla_a \phi(z)| \leq c^{-1}e^{-c\dist (z,  \Omega^{-}(\phi))}\,.
\end{equation}

\end{enumerate}
\end{theorem}

From this theorem we deduce the following important corollary, which we will use in the next section. It follows from the trivial observation that if $(a(z), \phi(z))$ is a solution to the vortex equations with $r=1$ and zeros at $\{z_j\}^{k}_{j=1}$, then
\begin{equation}\label{rv}
(a_r(z), \phi_r(z)):=(\sr a(\sr z), \phi(\sr z))
\end{equation}
is a solution to the rescaled vortex equations~\eqref{eq.vort}-\eqref{eq.vort2} with zeros at $\{\frac{z_j}{\sr}\}^{k}_{j=1}$. All the items in Corollary~\ref{corrol} then follow from Theorem~\ref{vortextaubes80} by rescaling according to 	Equation~\eqref{rv}.

\begin{corollary}\label{corrol}
Let $\cP:=\{z_j\}_{j=1}^k$ be a finite set of points $z_j\in\DD$, and let $\{m_j\}_{j=1}^k$ be an associated set of positive integers. For each $r>0$, there is a solution $(a_r, \phi_r)$ to the rescaled vortex equations~\eqref{eq.vort} and~\eqref{eq.vort2} on $\Cone$ with $|\phi_r|^{-1}(0)=\cP$ and with each zero $z_j$ having multiplicity $m_j$. Furthermore, the solution $(a_r, \phi_r)$ is bounded as $|\phi_r| < 1$ and has the following properties:
\begin{enumerate}
\item $|\nabla |\phi_r|^{2}(z)|\leq |\nabla_{a_r} \phi_r (z)| \leq C \sr$.
\item $
\cE_r:=r \int_\Cone (1-|\phi_r|^2)=2\pi \sum_j m_j$.
\item If we define
\[
\Omega^{-}_{r}:=\Big\{ z\in\Cone \text{ such that } |\phi_r|^{2}(z)\leq \frac{1}{2}\Big\} \,,
\]
there is a constant $c$ such that, if $\dist (z, \Omega^{-}_{r}) \geq \frac{1}{c\sr}$, we have
\[
|1-|\phi_{r}|^2(z)|\leq e^{-c \sr \dist (z, \Omega^{-}_{r}) }
\]
and
\[
|\nabla |\phi_{r}|^2(z)|\leq c^{-1} \sr e^{-c \sr  \dist (z, \Omega^{-}_{r})}
\]
\end{enumerate}

\end{corollary}

\subsection{Proof of Theorem \ref{T.3}}\label{SS.teo}

The theorem follows from the following key proposition, whose proof is relegated to Section~\ref{SS.keyprop}:

\begin{proposition}\label{rescaleddirac}
Let $\sigma_{\DD}$ be a probability measure on the disk. There is an increasing sequence of constants $r_n$ that tends to $\infty$, a sequence $\cP_{n}:=\{z_{jn}\}_{j=1}^{k_{n}} \subset \DD$ of finite sets of points, with $\{m_{jn}\}_{j=1}^{k_{n}}$ an associated collection of positive integers,  and a sequence of solutions $(a_{r_n},\phi_{r_n})$ to the Equations~\eqref{eq.vort} and~\eqref{eq.vort2} such that:
\begin{enumerate}
\item $\phi_{r_n}^{-1}(0)=\cP_{n}$, with multiplicities $\{m_{jn}\}$.

\item The sequence of measures
\[
\sigma_{n}:=\frac{r_n(1-|\phi_{r_n}|^{2}) dx \wedge dy}{\int_{\DD} da_{r_n}}
\]
converges weakly to $\sigma_{\DD}$.

\item As $n\rightarrow \infty$, we have
\[
\frac{\int_{\DD} d a_{r_n}}{ 2 \pi N_n} \rightarrow 1\,,
\]
where $N_{n}:=\sum_{j=1}^{k_n} m_{jn}$.

\item If $\sigma_\DD$ is $d$-Frostman for some $d>0$, then $N_n$
is bounded as
$$\lim_{n\to\infty}N_n r_{n}^{-\theta}=0\,,$$
with $\theta:=\min\big\{\frac{1}{4}, \, \frac{d}{2(d+1)}\big\}$.
\end{enumerate}
\end{proposition}

Let us first show how item~(i) in Theorem~\ref{T.3} follows from Proposition~\ref{rescaleddirac}. Given the sequence of solutions $(a_{r_n},\phi_{r_n})$, the discussion in Section~\ref{SS.vortex} shows that $\psi_r:=(\phi_{r_n},0)$ and $A_r:=a_{r_n}$ is a sequence of solutions of the modified Seiberg--Witten equations on $\cC$. Obviously $F_{A_r}=da_{r_n}$, $\la\wedge F_{A_r}=r_n(1-|\phi_{r_n}|^2)dx\wedge dy\wedge dt$ and the energy of the solutions is $\cE_r=\int_\DD da_{r_n}$. The sequence of measures of the Seiberg--Witten equations is then
\[
\sigma_n\otimes dt\,,
\]
and therefore item~(ii) above implies that it converges weakly to $\sigma_\DD\otimes dt$, as claimed. The energy $\cE_r$ is bounded as $N_{n}$, when $r\to\infty$, by item~(iii). Assuming that the measure $\sigma_\DD$ is $d$-Frostman, item~(iv) provides an estimate for $N_{n}$, which immediately implies item~(ii) in Theorem~\ref{T.3}, which completes the proof.

%

In the following proposition, we show the connection between the regularity of a measure and its Frostman properties alluded to in the Introduction. Recall that the Sobolev space $W^{-1,p}_{\DD}$ is defined as
\[
W^{-1,p}_{\DD}:=\left\{\phi\in W^{-1,p}(\RR^2): \supp(\phi)\subset \overline{\DD}\right\}\,.
\]
It easily follows from a duality argument and the Sobolev embedding theorem that any measure $\siD$ is in $W^{-1, p}_{\DD}$ for all $p<2$. This result is sharp, as evidenced by the Dirac measure $\delta_{p_0}$ supported at a point $p_0\in\DD$. When the measure is slightly more regular, we infer that it is $d$-Frostman for some $d>0$:

\begin{proposition}\label{P.frost}
Assume that the probability measure $\siD$ is in the Sobolev space $W^{-1,p}_{\DD}$ for some $p\in(2,\infty]$. Then $\sigma_\DD$ is $d$-Frostman with $d:=1-\frac2p$.
\end{proposition}
\begin{proof}
Take a smooth bump function $\chi:\RR^2\to[0,1]$ such that $\chi(x)=1$ if $|x|\leq 1$ and $\chi(x)=0$ if $|x|\geq2$. Defining $\chi_\ep(x):=\chi\big(\frac{x-x_0}{\ep}\big)$ for any point $x_0\in\RR^2$ and any $\ep>0$, we obviously have
\begin{equation}\label{CauchyS}
\siD(B(x_0,\ep))\leq \int_{\RR^2} \chi_\ep(x)\, d\siD(x)\,.
\end{equation}
Let $p':=p/(p-1)\in[0,2)$ be the dual exponent to~$p$. By scaling, the $W^{1,p'}$ norm of~$\chi_\ep$ satisfies
\begin{align*}
\|\chi_\ep\|_{W^{1,p'}}= \|\chi_\ep\|_{L^{p'}}+ \|\nabla\chi_\ep\|_{L^{p'}} \leq C\ep^{2-\frac2p}+ C\ep^{1-\frac2p}\leq C\ep^{1-\frac2p}\,.
\end{align*}
The generalized H\"older inequality then allows us to estimate~\eqref{CauchyS} as
\[
\siD(B(x_0,\ep))\leq \|\siD\|_{W^{-1,p}} \|\chi_\ep\|_{W^{1,p'}}\leq C\|\siD\|_{W^{-1,p}}\ep^{1-\frac2p}\,.
\]
The lemma then follows.
\end{proof}

\subsection{Proof of Proposition \ref{rescaleddirac}}\label{SS.keyprop}

The proof is divided in five steps. Items (i) and (iii) are established in Steps~2 and~4, respectively, while items (ii) and (iv) are proved in Step~5. The proof of some intermediate lemmas is postponed to Section~\ref{SS.lem}.

\subsubsection*{Step 1: Choice of a sequence of points}

We claim that we can choose a sequence of finite sets of points $\cP_n=\{z_{jn}\}_{j=1}^{k_n} \subset \DD$ with multiplicities $\{m_{jn}\}_{j=1}^{k_n}$, $n \in \NN$, and $N_{n}=\sum_j m_{jn}$, such that the Dirac measures
\[
\delta_{\cP_n}:=\frac{1}{N_n} \sum_{j=1}^{k_n} m_{jn}\delta(z-z_{jn})
\]
converge weakly to $\sigma_{\DD}$ as $n \rightarrow \infty$, and, moreover, if we define
\begin{equation}\label{eq.epn}
\epsilon_{n}:=\frac{1}{2} \min \bigg(\min_{j, k} |z_{jn}-z_{kn}|, \, \min_{j} \dist(z_{jn}, \partial \DD)\bigg)\,.
\end{equation}
we have:
\begin{enumerate}

\item There is a decreasing continuous function $F: (0, \infty) \rightarrow (0, \infty)$ with
\[
\lim_{x\rightarrow \infty} F(x)=0
\]
and so that
\[
\epsilon_n \geq F(N_{n})\,.
\]

\item If $\sigma_{\DD}$ is $d$-Frostman with $d>0$, this function can be taken $F(x):=Cx^{-\frac{1}{d}}$ for some constant $C>0$, so that
\[
\epsilon_{n} \geq \frac{C}{ N_{n}^{\frac{1}{d}}} \,.
\]
\end{enumerate}

Indeed, in the case that $\sigma_\DD$ is a point measure there is $n_0\geq 1$ such that $\cP_n=\cP_{n_0}$ for all $n\geq n_0$, and the same with $N_n$ and $\epsilon_n$; the claim is then obvious because the function $F$ can be chosen so that $F(N_n)\leq \epsilon_{n_0}$ for all $n\geq n_0$. Otherwise, it is standard that we can always approximate the measure $\sigma_{\DD}$, in the sense of weak convergence, by a sequence of Dirac probability measures of the form
\[
\Sigma_n:=\frac{1}{M_n}\sum_{j=1}^{k_n} \sigma_{\DD}(B_{\epsilon'_n}(z_{jn})) \delta(z-z_{jn})
\]
for some sequence of points $\{z_{jn}\}_{j=1}^{k_n}\subset\DD$. Here, $\{B_{\epsilon'_n}(z_{jn})\}$ is a disjoint collection of balls of radius $\epsilon'_n$ inside the disk, which cover it when $n\to\infty$ (and hence $\epsilon'_n\to 0$), and $M_n:=\sum_{j=1}^{k_n} \sigma_{\DD}(B_{\epsilon'_n}(z_{jn}))$. By density, we can safely assume that each $\sigma_{\DD}(B_{\epsilon'_n}(z_{jn}))M_n^{-1}$ is a rational number of the form
\[
\frac{\sigma_{\DD}(B_{\epsilon'_n}(z_{jn}))}{M_n}=\frac{m_{jn}}{N_n}
\]
for some positive integers $m_{jn}$ and $N_n$. Obviously, $\sum_j m_{jn}=N_n$, $k_n\leq N_n$ and $M_n\to 1$ as $n\to\infty$. The measure $\Sigma_n$ is then of the form $\delta_{\cP_n}$ stated above. Moreover, the numbers $\epsilon_n$ defined in Equation~\eqref{eq.epn} are bounded from below as $\epsilon_n\geq\epsilon'_n$.

If $\sigma_{\DD}$ is $d$-Frostman, then $\sigma_{\DD}(B_{\epsilon}(x)) \leq C \epsilon^{d}$, so taking $n$ large enough so that $M_n\geq \frac12$, we deduce the relation
\[
N_n\epsilon_n^d\geq N_n (\epsilon'_n)^d\geq \frac{C}{2} \,,
\]
which proves the item~(ii) above. If the measure $\sigma_\DD$ is not $d$-Frostman, there is no explicit relation between $\epsilon_n$ and $N_n$. However, using that the sequence $\epsilon_n$ can be chosen to be decreasing, we can always define $F(N_n):= \epsilon_n$, and find a decreasing positive function $F$ interpolating those values such that $\lim_{x\to\infty}F(x)=0$, so that item~(i) is trivially true.

In what follows we fix a sequence $\cP_n$ and associated multiplicities $\{m_{jn}\}_{j=1}^{k_n}$ with the properties stated above. In particular, if the measure is $d$-Frostman, we assume that $F$ is given as in item (ii).


\subsubsection*{Step 2. Choice of a sequence of rescaled vortex solutions} By Corollary~\ref{corrol}, for any sequence of positive real numbers $\{r_{n}\}_{n=0}^{\infty}$ we can find a sequence of solutions $(a_{r_n}, \phi_{r_n})$ to the $r_{n}$-rescaled vortex equations on $\CC$, with $\phi_{r_n}^{-1}(0)=\cP_{n}$, and with associated multiplicities $\{m_{jn}\}$. This already proves item (i) of the proposition.

For the rest of the proof, it is convenient to fix a sequence $\{r_{n}\}$ so that the following conditions are satisfied as $n\rightarrow \infty$:
\begin{equation}\label{rn}
0=\lim_{n\to\infty}\frac{N_n}{r^{\frac{1}{4}}_{n}}  =\lim_{n\to\infty} \frac{N_n}{F(N_n) \sqrt{r_n}}=\lim_{n\to\infty} \frac{\log r_{n}}{F(N_n) \sqrt{r_n}}\,,
\end{equation}
where $F$ is the function defined in Step~$1$.  Observe that such a sequence always exists because it suffices to take $r_n$ large enough for each $n$.

It is easy to see that in the case that $\sigma_{\DD}$ is $d$-Frostman for some $d>0$, and so $F(N_n)=C N_n^{-\frac{1}{d}}$, Equation~\eqref{rn} is satisfied if we choose a sequence $r_n$ verifying
\begin{equation}\label{eq.Frost}
\lim_{n\to\infty}N_nr^{-\theta}_{n}= 0
\end{equation}
with
\[
\theta:=\min\bigg\{\frac{1}{4},\, \frac{d}{2(d+1)}\bigg\}\,.
\]

\subsubsection*{Step 3: Some key auxiliary Lemmas}

It is convenient to define
\begin{align*}
	\Omega_{n}^{+}&:=\DD \setminus \bigcup_{j} B(z_{jn}, F(N_n))\,,\\
\Omega_{n}^{-}&:=\left\{z \in \DD : |\phi_{r_n}(z)|^{2} < \frac{1}{2}\right\}
\end{align*}
and to let $\Omega_{n}^{-}(z_{jn})$ denote the connected component of $\Omega_{n}^{-}$ which contains $z_{jn}$. Note that
\[
 \Om_n^- = \bigcup_j \Om_n^-(z_{jn})
\]
because there is a zero of~$\phi$ in each connected component of~$\Om_n^-$ (simply because, by the form of the vortex equations, $\phi$ must vanish at each minimum of~$|\phi|^2$).

Our goal is to show that, as $n\to\infty$, the function $|\phi_{r_n}|$ is exponentially close to $1$ in the set $\Omega_{n}^{+}$, while in the set $\Omega_{n}^{-}$ the solution goes to zero with a polynomial bound.
Before stating the lemmas that establish these properties, we notice that Equation~\eqref{rn} implies, for any fixed constant $C$ independent of $n$, that
\begin{equation}\label{littleball}
B\left(z_{jn}, \frac{C N_n}{\sqrt{r_n}}\right) \subset B(z_{jn}, F(N_n))
\end{equation}
provided that $n$ is large enough. The proofs of these lemmas will be presented in Section~\ref{SS.lem}.

\begin{lemma}\label{balls}
$\Omega^{-}_{n}(z_{jn}) \subset B(z_{jn}, \frac{CN_n}{\sqrt{r_n}})$ for some constant $C$ independent of $n$, for all $j$ and all large enough $n$.
\end{lemma}

\begin{lemma}\label{outsid}
For any $z \in \Omega_{n}^{+}$ and all $n>0$, the solution to the vortex equations is bounded as
\[
1-|\phi_{r_n}(z)|^2 \leq e^{-c F(N_n) \sqrt{r_n}}\,,
\]
\[
|\nabla |\phi_{r_n}|^2| \leq c^{-1} \sqrt{r_n} e^{-c F(N_n) \sqrt{r_n}}\,.
\]
Moreover, for any $z \in \CC \setminus \DD$ the estimate is
\[
1-|\phi_{r_n}(z)|^2 \leq e^{-c\sqrt{r_n}(||z|-1|+F(N_n))}\,,
\]
\[
|\nabla |\phi_{r_n}|^2 | \leq c^{-1} \sqrt{r_n} e^{-c\sqrt{r_n} (||z|-1|+F(N_n))}\,.
\]
Here $c$ is a constant that does not depend on $n$.
\end{lemma}

Observe that the third condition in Equation~\eqref{rn} implies that $F(N_n) \sqrt{r_n} \rightarrow \infty$ as $n\rightarrow \infty$ faster than $\log r_n$ (even in the case where $N_n$ is constant for all $n\geq n_0$), so all the upper bounds in Lemma~\ref{outsid} go to $0$ as $n \rightarrow \infty$.

\begin{lemma}\label{insid}
For some constant $C>0$ (independent of $n$), on each ball $B(z_{jn}, \frac{C N_n}{\sqrt{r_n}})$ we can write
\[
|\phi_{r_n}(z)|^2=r^{m_{jn}}_{n} h_{jn}(z) |z-z_{jn}|^{2m_{jn}}\,,
\]
where $h_{jn}(z)$ is a smooth function satisfying $h_{jn}(z)>0$ and
\[
\bigg|\frac{1}{N_n} \sum_{j} \int_{B(z_{jn}, \frac{CN_n}{\sqrt{r_n}})}\log (h_{jn}) \bigg| \rightarrow 0
\]
as $n \rightarrow \infty$.
\end{lemma}

\subsubsection*{Step 4: Proof of item (iii)}

Taking into account item~(ii) in Corollary~\ref{corrol}, it is enough to show that
\[
\lim_{n\to\infty} r_{n}\int_{\CC \setminus \DD}(1-|\phi_{r_n}|^2) = 0\,.
\]
By Lemma \ref{outsid}, we have the estimate
\[
\int_{\CC \setminus \DD} r_n(1-|\phi_{r_n}|^{2}) \leq r_n e^{-c \sqrt{r_n} F(N_n)} \int_{\CC \setminus \DD} e^{-c \sqrt{r_n} |z-1|}\,.
\]
Accordingly, since $F(N_n) \sqrt{r_n} \rightarrow \infty$ faster than $\log r_n$, the claim follows.

\subsubsection*{Step 5: Proof of items (ii) and (iv)}

Once item~(iii) has been established, to prove item~(ii) it is enough to show that, for any function $f\in C^\infty(\overline\DD)$,
\[
\int_{\DD} r_{n}(1-|\phi_{r_n}|^2)  f =2 \pi \sum_{z_{jn} \in \cP_n} m_{jn}f(z_{jn})+e(r_n)\,,
\]
with an error satisfying  $\lim_{n\to\infty}{e(r_n)}/{N_n} = 0$.

It will be convenient to work with the function $u_{r_n}$ defined as
\[
 u_{r_n}:=\log |\phi_{r_n}|^2.
\]
Since $|\phi_{r_n}|< 1$ (cf. Corollary~\ref{corrol}), the function $u_{r_n}$ is negative. It is not hard to check that the function $u_{r_n}$ satisfies, as a distribution, the PDE (see e.g.~\cite[Chapter 3.3]{JT})
\begin{equation}\label{ur}
\Delta u_{r_n} +2 r_n (1-e^{u_{r_n}})=4\pi \underset{z_{j_n} \in \cP_n}{\sum} m_{jn} \delta (z-z_{jn})\,.
\end{equation}
In terms of $u_{r_n}$, the measure $\sigma_{n}$ reads as
\[
\sigma_{n}=\frac{r_n(1-e^{u_{r_n}}) dx \wedge dy}{\int_{\DD} da_{r_n}}\,.
\]
Noticing that Equation~\eqref{ur} implies that, for any $f \in C^{\infty}(\overline\DD)$,
\[
r_n \int_{\DD} (1-e^{u_{r_n}}) f =-\frac{1}{2}\int_{\DD} f \Delta u_{r_n}  + 2\pi \underset{ \cP_n}{\sum}  m_{jn}f(z_{jn})\,,
\]
we infer that item~(ii) follows if we prove that
\[
\lim_{n\to\infty}\frac{\bigg|\int_{\DD} f \Delta u_{r_n} \bigg|}{N_n} =0\,,
\]
for all $f\in C^\infty(\overline\DD)$.

To this end, we first integrate by parts to obtain
\begin{equation}\label{laplaceu}
\int_{\DD} f \Delta u_{r_n}= -\int_{\DD} \nabla u_{r_n} \cdot  \nabla f+\int_{\partial \DD} f \nabla u_{r_n} \cdot \nu \, d\theta\,,
\end{equation}
where $\nu$ is the outward pointing unit normal vector at the boundary of the disk. Now, by Lemma~\ref{outsid}, for any point $z\in\partial\DD$  and all $n$ large enough, we have the estimate
\[
|\nabla u_{r_n}|(z)=\frac{1}{|\phi_{r_n}(z)|^{2}}|\nabla |\phi_{r_n}(z)|^{2}| \leq C \sqrt{r_n} e^{-c F(N_n) \sqrt{r_n}}\,.
\]
Therefore,
\[
\left|\int_{\partial \DD} f \nabla u_{r_n}\cdot\nu\right| \leq C\|f\|_{L^1(\partial\DD)} \sqrt{r_n} e^{-c F(N_n) \sqrt{r_n}}\,,
\]
which goes to zero as $n\to\infty$ because $F(N_n) \sqrt{r_n}$ tends to infinity faster than $\log(r_n)$ (cf. Equation~\eqref{rn}).

As for the first summand in Equation~\eqref{laplaceu}, a second integration by parts yields
 \[
- \int_{\DD} \nabla u_{r_n} \cdot  \nabla f= \int_{\DD} u_{r_n} \Delta f-\int_{\partial \DD} u_{r_n} \nabla f \cdot \nu d\theta\,.
 \]
Again, using Lemma~\ref{outsid}, the rightmost term is bounded as
 \[
\bigg|\int_{\partial \DD} u_{r_n} \nabla f \cdot \nu d\theta \bigg| \leq C\|f\|_{W^{1,1}(\partial\DD)} e^{-c F(N_n) \sqrt{r_n}}\,,
 \]
which again goes to zero as $n\to\infty$. Finally, using that $u_{r_n}<0$, it is clear that
 \[
 -\int_{\DD} u_{r_n} \Delta f \leq -\|f\|_{C^{2}(\DD)} \int_{\DD} u_{r_n} \,,
 \]
so our main claim follows if we show that
\begin{equation}\label{eq.ninf}
\lim_{n\to\infty}\frac{1}{N_n}\int_{\DD} u_{r_n} =0\,.
\end{equation}

To prove this, we divide the integral into two parts, the disks $B(z_{jn}, F(N_n))$, and the set $\Omega^{+}_{n}$:
\[
-\int_{\DD} u_{r_n}=-\int_{\Omega_{n}^{+}} u_{r_n}-\sum_{\cP_n} \int_{B(z_{jn}, F(N_n))} u_{r_n} \,.
\]
By Lemma \ref{outsid}, for any $z \in \Omega^{+}_{n}$ we can write the bound
\[
|1-|\phi_{r_n}(z)|^{2}| \leq e^{-cF(N_n)\sqrt{r_n}} \ll 1
\]
provided that $n$ is large enough. Thus, taking the Taylor expansion of
\[
u_{r_n}={\log  }(1-(1-|\phi_{r_n}|)^2)
\]
we obtain
\begin{multline*}
-\int_{\Omega_{n}^{+}} u_{r_n}=-\int_{\Omega_{n}^{+}} \log |\phi_{r_n}|^2 \leq \int_{\Omega_{n}^{+}} (1-|\phi_{r_n}|)^2+C \int_{\Omega_{n}^{+}} (1-|\phi_{r_n}|^2)^2  \\ \leq C e^{-c F(N_n) \sqrt{r_n}}\,,
\end{multline*}
which tends to $0$ as $n\to\infty$.

To bound the integral
\[
-\int_{B(z_{jn}, F(N_n))} u_{r_n}\,,
\]
we write it for large enough $n$ as
\[
-\int_{B(z_{jn}, F(N_n))} u_{r_n}=-\int_{B(z_{jn}, \frac{CN_n}{\sqrt{r_n}})} u_{r_n}-\int_{B(z_{jn}, F(N_n))\setminus B(z_{jn}, \frac{CN_n}{\sqrt{r_n}}) } u_{r_n}\,,
\]
where $C$ is the constant in Lemma~\ref{balls} and we have used Equation~\eqref{littleball}. By Lemma~\ref{balls}, we know that $|\phi_{r_n}|^2 \geq \frac{1}{2}$ on the set $B(z_{jn}, F(N_n))\setminus B(z_{jn}, \frac{CN_n}{\sqrt{r_n}})$, so on this set
\[
0\leq -u_{r_n} \leq \log 2\,,
\]
which allows us to write the bound
\begin{equation}\label{anillos}
-\int_{B(z_{jn}, F(N_n))\setminus B(z_{jn}, \frac{CN_n}{\sqrt{r_n}}) } u_{r_n} \leq C F(N_n)^2\,.
\end{equation}
On the other hand, by Lemma~\ref{insid}, we can express $u_{r_n}$ in the disk $B(z_{jn}, \frac{C N_n}{\sqrt{r_n}})$ as
\[
u_{r_n}(z)= \log(h_{jn})+m_{jn}  \log r_n+2 m_{jn} \log(|z-z_{jn}|)\,,
\]
so we deduce
\begin{multline*}
-\int_{B(z_{jn}, \frac{CN_n}{\sqrt{r_n}})} u_{r_n} \leq -\int_{B\big(z_{jn}, \frac{CN_n}{\sqrt{r_n}}\big)} \bigg[\log(h_{jn})-\pi m_{jn}\bigg(\frac{CN_n}{\sqrt{r_n}}\bigg)^2 \log r_n \\-2\pi m_{jn} \bigg(\frac{CN_n}{\sqrt{r_n}}\bigg)^2 \log  \bigg(\frac{CN_n}{\sqrt{r_n}}\bigg)+\pi m_{jn}\bigg(\frac{CN_n}{\sqrt{r_n}}\bigg)^2\bigg]\,.
\end{multline*}
The first term after the inequality divided by $N_n$ goes to zero because of Lemma~\ref{insid}, so putting the other terms together with the one coming from Equation~\eqref{anillos}, Equation~\eqref{eq.ninf} follows if we show that the quantity
\begin{align*}
 \frac{1}{N_n} \sum_{\cP_n} \bigg[C F(N_n)^2-\pi m_{jn} \bigg(\frac{CN_n}{\sqrt{r_n}}\bigg)^2 \log r_n  \\-2\pi m_{jn} \bigg(\frac{CN_n}{\sqrt{r_n}}\bigg)^2 \log  \bigg(\frac{CN_n}{\sqrt{r_n}}\bigg)
+\pi m_{jn}\bigg(\frac{CN_n}{\sqrt{r_n}}\bigg)^2 \bigg]\\\leq C\bigg( F(N_n)^2+ \bigg(\frac{N_n}{\sqrt{r_n}}\bigg)^2 \log r_n  + \bigg(\frac{N_n}{\sqrt{r_n}}\bigg)^2 \log  N_n\bigg)
\end{align*}
goes to zero as $n \rightarrow \infty$. But this is evident, because, by construction, if $\sigma_{\DD}$ is not a point measure,
$$\lim_{n\to\infty}F(N_n)=0\,,$$
and the other terms also tend to $0$ as $n\to\infty$ by the conditions in Equation~\eqref{rn}. If $\sigma_{\DD}$ is a point measure, and thus $N_{n}$ stays constant for all $n\geq n_0$, we reach the same conclusion by substituting in the argument above the sequence $F(N_n)$ by a sequence $F_{n}$ of positive numbers, smaller than $\epsilon_{n_0}$, and going to zero as $n \rightarrow \infty$.

This completes the proof of item~(ii). Concerning item~(iv), we simply recall that the condition~\eqref{rn} is verified by any $d$-Frostman measure upon choosing a sequence of $r_n$ satisfying Equation~\eqref{eq.Frost}. Proposition~\ref{rescaleddirac} then follows.

\subsection{Proof of the auxiliary lemmas}\label{SS.lem}

In this section we prove Lemmas~\ref{balls},~\ref{outsid} and~\ref{insid}, which are instrumental in the previous section. We follow the same notation and assumptions as before without further mention.

\subsubsection{Proof of Lemma \ref{balls}}

The claim obviously follows if we show that for any points $p_n$ and $q_n$ in the same connected component of $\Omega^{-}_{n}$ there is a constant $C$ (independent of $n$) such that
\[
\dist (p_n, q_n)\leq \frac{CN_n}{\sqrt{r_n}}\,.
\]

Indeed, let $\gamma_{n}$ be a smooth embedded curve inside $\Omega^{-}_{n}$, joining the points $p_n$ and~$q_n$ (which exists because $\Omega^{-}_{n}$ is an open set). By definition, any point  $z \in \gamma_{n}$ satisfies that $|\phi_{r_n}|^{2}(z)< \frac{1}{2}$. Using that
\[
|\nabla |\phi_{r_n}|^{2}(z)|\leq C \sqrt{r_n}
\]
for all $z\in\CC$ by Corollary~\ref{corrol}, we infer that there is a constant $C$ (independent of $n$) such that, for any $\delta>0$ as small as desired, all the points within a distance $\frac{\delta}{C\sr_n}$ of $\gamma_{n}$ satisfy $|\phi_{r_n}|^{2}\leq \frac{1}{2}+\delta$.

Fixing a small constant $\delta>0$, let us denote by $U_n$ the aforementioned set of points $z\in \Cone$ at a distance smaller or equal than $\frac{\delta}{C\sr_n}$ from $\gamma_{n}$. The area of $U_n$ is bounded from below by
\[
\int_{U_n} dx\wedge dy \geq |\gamma_{n}| \frac{\delta}{C\sr_n}\,,
\]
where by $|\gamma_n|$ we denote the length of the curve $\gamma_n$. Therefore,
\[
r_n\int_{U_n} (1-|\phi_{r_n}|^{2}) \geq r_n |\gamma_n| \Big(\frac{1}{2}-\delta \Big) \frac{\delta}{C\sr_n} \geq C\sr_n |\gamma_n|\,,
\]
for some constant $C$ that depends on $\delta$ but not on $n$. On the other hand, notice that, by Corollary~\ref{corrol},
\[
r_n\int_{U_n} (1-|\phi_{r_n}|^{2}) \leq r_n \int_{\Cone} (1-|\phi_{r_n}|^{2}) =2\pi N_n\,.
\]
Therefore, combining both inequalities we can bound the length of $\gamma_n$ as
\[
|\gamma_{n}| \leq C\frac{N_n}{\sr_n}\,
\]
for some $n$-independent constant $C$. The claim follows because the length of $\gamma_n$ is always greater or equal than the distance between $p_n$ and $q_n$.

\subsubsection{Proof of Lemma \ref{outsid}}
We recall Equation~\eqref{littleball}, i.e.,
$$B(z_j, \frac{CN_n}{\sr_n}) \subset B(z_{jn}, F(N_n))\,.$$
Then Lemma~\ref{balls} implies that
\[
\Omega_{n}^{-}(z_{jn}) \subset B(z_{jn}, F(N_n))\,.
\]
In particular, since $\epsilon_n\geq F(N_n)$ by definition, it follows that
\[
\Omega_{n}^{-}(z_{jn})  \cap \Omega_{n}^{-}(z_{kn})=\emptyset
\]
for any $j \neq k$. Let us estimate the infimum of $\dist (z, \Omega_{n}^{-})$ for $z \in \Omega_{n}^{+}$. It is clear that we can take $z$ on the boundary $\partial B(z_{jn}, F(N_n))$ for some $j$, in which case Lemma~\ref{balls} implies that, for some $C>0$,
\[
\dist (z, \Omega_{n}^{-})  \geq F(N_n)-\frac{C N_n}{\sqrt{r_n}}=F(N_n) \bigg(1-\frac{C N_n}{F(N_n)\sqrt{r_n}}\bigg) \,.
\]
Then, we deduce from Equation~\eqref{rn} that for any given $\delta>0$ and any large enough $n$ we have
\[
\dist (z, \Omega_{n}^{-})  \geq (1-\delta) F(N_n) \gg\frac{1}{\sqrt{r_n}}\,,
\]
where we have used that $\lim_{n\to\infty}F(N_n)\sqrt{r_n}=\infty$. The first two statements in Lemma~\ref{outsid} then follow by applying item~(iii) of Corollary~\ref{corrol}.

The two other statements concerning points $z \in \CC \setminus  \DD$ also follow from item (iii) in Corollary \ref{corrol} upon noticing that
\[
\dist (z, \Omega_{n}^{-}) \geq \dist (z, \partial \DD)+\dist (\partial \DD, \Omega_{n}^{-}) \geq ||z|-1|+ (1-\delta) F(N_n)\,.
\]
This completes the proof of the lemma.

\subsubsection{Proof of Lemma \ref{insid}}
It is well known, cf.~\cite[Proposition~5.1]{JT}, that any solution $\phi$ for the $r=1$ vortex equations can be written as
\[
\phi(z)=(h_k (z))^{1/2} (z-z_k)^{m_k}
\]
on a disk that contains just one zero $z_{k}$. Here $h_k(z)$ is a smooth non-vanishing function on the disk and $m_k$ is the multiplicity of the zero. By rescaling, we get the first statement in Lemma \ref{insid}, that is, we can represent $\phi_{r_n}$ as
\begin{equation}\label{hdef}
|\phi_{r_n}(z)|^2=h_{jn}(z) r^{m_{jn}}_n |z-z_{jn}|^{2m_{jn}}\,.
\end{equation}
Notice that this representation holds on the disk $B(z_{jn},\frac{CN_n}{\sqrt{r_n}})$ for some constant $C>0$ because, by construction, $\epsilon_n\geq F(N_n)\geq \frac{CN_n}{\sqrt{r_n}}$, and hence $z_{jn}$ is the only zero of $\phi_{r_n}$ in such a disk.

For notational simplicity, we define the smooth function $v_{r_n}: B(z_{jn}, \frac{CN_n}{\sqrt{r_n}})\rightarrow \RR$ as
\[
v_{r_n}(z):=\log h_{jn}(z)\,,
\]
and we set $B_{n}:= B(z_{jn},\frac{CN_n}{\sqrt{r_n}})$.

To prove the estimate for $h_{jn}$ in Lemma~\ref{insid}, we first notice that
\begin{equation}\label{vineq1}
\bigg|\int_{B_n} v_{r_n}\bigg| \leq \pi^{\frac{1}{2}} \frac{C N_n}{\sqrt{r_n}} \|v_{r_n}\|_{L^2(B_n)}\,.
\end{equation}
Our goal is to bound the $L^2$ norm $\|v_{r_n}\|_{L^2(B_n)}$. To this end, we first observe that, if $w_{r_n}$ is the unique harmonic function on the disk $B_n$ that coincides with $v_{r_n}$ at the boundary, we have the inequality
\begin{equation}\label{vineq2}
\|v_{r_n}\|_{L^2(B_n)} \leq \|w_{r_n}\|_{L^{2}(B_n)}+\frac{1}{\lambda_{1}(B_n)} \|\Delta v_{r_n}\|_{L^2(B_n)}\,
\end{equation}
where $\lambda_{1}(B_{n})=\frac{c_0r_n}{N_n^2}$ is the first eigenvalue of the Dirichlet Laplacian on the disk $B_n$ (for some constant $c_0$). This estimate follows easily from the min--max characterization of Dirichlet eigenvalues.

Now, the maximum principle for harmonic functions allows us to write
\[
\sup_{B_n} |w_{r_n}|=\sup_{\partial B_n} |w_{r_n}|=\sup_{\partial B_n} |v_{r_n}|\,,
\]
and therefore,
\[
\|w_{r_n}\|_{L^2(B_n)} \leq \frac{CN_n}{\sqrt{r_n}}\,\sup_{\partial B_n} |v_{r_n}|\,.
\]
To obtain a bound of $\sup_{\partial B_n} |v_{r_n}|$, we recall that $\Omega_{n}^{-}(z_{jn}) \subset B_n$ by Lemma~\ref{balls}, so $|\phi_{r_n}|^2\geq \frac12$ on $\partial B_n$, which implies by Equation~\eqref{hdef}
\[
h_{jn}|_{\partial B_n} \geq \frac{1}{2} \frac{1}{(C N_n)^{2 m_{jn}}}\,,
\]
and hence
\[
v_{r_n}|_{\partial B_n} \geq -\log 2-2m_{jn}\log(C N_n)\,.
\]
On the other hand, $|\phi_{r_n}|^2 \leq 1$, so applying again Equation~\eqref{hdef} and taking the logarithm, we get the upper bound
\[
v_{r_n}|_{\partial B_n}  \leq - 2m_{jn}\log(C N_n)\,.
\]
We then conclude that
\[
\sup_{\partial B_n} |v_{r_n}| \leq |2m_{jn}\log(C N_n)+1| \leq 2m_{jn}|\log(C N_n)+1|\,,
\]
and therefore
\begin{equation}\label{vineq3}
\|w_{r_n}\|_{L^2(B_n)} \leq  \frac{C N_n m_{jn}}{\sqrt{r_n}} |\log(C N_n)+1|\,.
\end{equation}

Finally, to obtain a bound for the $L^2$ norm of $\Delta v_{r_n}$, we use Equation~\eqref{ur} and the fact that
\[
\Delta \log |z-z_{jn}|^{2m_{jn}}=4 \pi m_{jn} \delta(z-z_{jn})\,,
\]
to infer that the smooth function $v_{r_n}$ satisfies the PDE
\begin{equation}\label{vr}
\Delta v_{r_n}=2 r_{n}(e^{u_{r_n}}-1)\,.
\end{equation}
Accordingly,
\[
\|\Delta v_{r_n}\|_{L^2(B_n)}=2 r_{n} \bigg(\int_{B_n} |e^{u_{r_n}}-1|^2 \bigg)^{\frac{1}{2}}\,,
\]
and since $|\phi_{r_n}|^2=e^{u_{r_n}} < 1$, we obtain the estimate
\begin{equation}\label{vineq4}
\|\Delta v_{r_n}\|_{L^2(B_n)} \leq C N_n \sqrt{r_n}\,.
\end{equation}
Putting together Equations~\eqref{vineq1},~\eqref{vineq2},~\eqref{vineq3} and~\eqref{vineq4} we get the bound
\begin{align}\label{vineq5}
\bigg|\int_{B_n} v_{r_n}\bigg| \leq  \frac{Cm_{jn}N^{2}_n}{r_n} |\log(C N_n)+1|+\frac{CN^{4}_n}{r_n}  \,,
\end{align}
and finally, using that $N_n=\sum_{\cP_n} m_{jn}$ and $k_n \leq N_n$, we obtain
\begin{align*}
\frac{1}{N_n} \sum_{\cP_n} \bigg|\int_{B_n} v_{r_n}\bigg| \leq  \frac{1}{N_n} \sum_{\cP_n} m_{jn} \bigg(\frac{CN^{2}_n}{r_n} |\log(C N_n)+1|\bigg)+\frac{1}{N_n} \sum_{\cP_n} \frac{CN^{4}_n}{2 r_n} \\ \leq \frac{CN^{2}_n}{r_n} |\log(C N_n)+1|+\frac{C N^{4}_n}{r_n}
\end{align*}
which goes to zero as $n\to\infty$ by the way the sequence of $r_n$ was constructed, cf. Equation~\eqref{rn}. This completes the proof of the lemma.

\section{Energy growth and ergodicity}\label{S.ergodic}

In this final section we include a simple observation on the limiting invariant measures that one obtains when the energy growth of the sequence of solutions to the modified Seiberg--Witten equations is linear. By this we mean that there exists a positive constant $C>0$, independent of $n$, such that
\[
C^{-1} r_n\leq \cE_n\leq C r_n\,.
\]

\begin{theorem}\label{T.ergodic}
Let $(r_{n},\,\psi_{n},\,A_{n})_{n=1}^\infty$ be a sequence of solutions to the modified Seiberg--Witten equations as in Theorem~\ref{T.taubes}. If the energy sequence $\cE_n$ has linear growth, then the vector field~$X$ cannot be ergodic (with respect to the Lebesgue measure).

\end{theorem}

\begin{proof}
By Equation~\eqref{sigma_n}, the signed measures $\sigma_n$ can be written as
\[
\sigma_{n}(U)=\frac{r_n\int_{U}(1-|\al_n|^2) \mu}{\cE_n}+O(\cE^{-1}_n)
\]
for any domain $U \subset M$.
Accordingly, if the energy growth is linear, we obtain
\[
\sigma_{n}(U) \leq C \int_{U}|1-|\al_n|^2| \mu + O(r_{n}^{-1}) \leq C\mu(U)\,,
\]
where we have used that $|\alpha_n|$ is uniformly bounded. Taking the limit $n\to\infty$, this implies that $\sigma_{\infty}(U)=0$ whenever $\mu(U)=0$. In other words, $\sigma_{\infty}$ is absolutely continuous with respect to $\mu$.

Then, it is well known that we can write $\sigma_{\infty}=f \mu$, where $f \in L^1(M)$ is the Radon-Nikodym derivative of $\sigma_{\infty}$ with respect to $\mu$. Since both $\sigma_{\infty}$ and $\mu$ are invariant measures, $f$ can be understood as an $L^1$ function that is invariant under the flow of $X$. Therefore, if~$X$ is ergodic, the ergodic theorem implies that $f$ is constant, i.e., $f=\int_Mf\mu=1$, at almost every point of $M$.

However, the main observation is that $f$ cannot be a.e. constant, because item (ii) in Theorem \ref{T.taubes} ensures that, for any 1-form $\gamma$ such that $d\gamma=i_{X} \mu$:
\[
\int_M \star(\gamma \wedge d \gamma)f\mu=\sigma_{\infty}(\star(\gamma \wedge d \gamma)) \leq 0\,,
\]
while $\int_M\star(\gamma \wedge d \gamma)\mu=\cH(X)>0$ by hypothesis. This contradiction shows that $X$ cannot be ergodic.
\end{proof}

\section*{Acknowledgments}

This work has received funding from the European Research Council (ERC) under the European Union's Horizon 2020 research and innovation programme through the grant agreement~862342 (A.E.). It is partially supported by the grants CEX2019-000904-S, RED2018-102650-T, and PID2019-106715GB GB-C21 (D.P.-S.) funded by MCIN/AEI/10.13039/501100011033, and a Fields Ontario Postdoctoral Fellowship (F.T.L.) financed by the NSERC grant RGPIN-2019-05209. F.T.L. also wishes to thank the Max Planck Institute for Mathematics for its hospitality and financial support during part of this work.

\appendix

\section{The Seiberg--Witten invariant measures via foliation cycles}\label{S.foliated}

In this appendix, which is of independent interest, we revisit Taubes's Theorem~\ref{T.taubes} from the viewpoint of Sullivan's theory of foliated cycles, and refine it showing the property stated in Remark~\ref{R.=0}. Let us first recall some concepts.

\subsection{Preliminaries} We denote by $\cZ^{p}$ the space of $p$-currents on $M$, i.e., the continuous dual of the space of smooth $p$-forms. Notice that any $2$-form $\Theta$ on $M$ can be identified with a $1$-current (denoted in the same way) as follows: for any $1$-form $\theta$, the action of $\Theta$ on $\theta$ is given by $\int_M \Theta\wedge\theta$. Let $\mathcal{Z}_{X}$ and $\mathcal{C}_X$ be the set of foliation currents and of foliation cycles of the vector field $X$, respectively. We recall (see e.g.~\cite{SU76}) that a \emph{foliation current} of a vector field $X$ is a \mbox{1-current} that can be approximated arbitrarily well (in the weak topology) by $1$-currents supported on segments of orbits of the vector field. Equivalently, a foliation current can be approximated by $1$-currents of the form
\[
\sum_{i=1}^N c _i\, \delta^{p_i}_X\,,
\]
with $N \in \mathbb{N}$, $c_i \in [0, \infty)$ and $p_i \in M$, and where for any $p\in M$ the $1$-current $\delta^{p}_X$ is defined as
\[
\delta^{p}_X(\theta)=\theta_{p}(X) \text{ for any 1-form $\theta$} \,.
\]
A \emph{foliation cycle} is a closed foliation current, i.e., a foliation current whose kernel contains the linear subspace of exact 1-forms. It is straightforward to see that foliation cycles are in one to one correspondence with invariant measures of the vector field $X$.

Following~\cite{Ana}, associated to the vector field $X$, we also define the subset $\cF_X$ of the space of 1-currents, consisting of the boundaries of zero-flux surfaces, i.e.,
$$
\cF_X =\Big\{\partial S \,:\, S \, \text{ is a surface with $\int_{S} i_X \mu=0$}\Big\}\,.
$$
Notice that a 1-current $c$ being in $\overline{\cF_{X}}$ (where the closure is taken with respect to the weak topology in the space of 1-currents) is equivalent to the fact that, for any $\gamma$ such that $d \gamma=i_{X} \mu$
\[
c(\gamma)=0\,.
\]
In particular, if the 1-current $c$ is a smooth 2-form, this means that there is a $1$-form $b$ with $d b=c$ and
\[
\int_M b \wedge i_{X} \mu =0\,.
\]

Before stating the main theorem of this appendix, following the same notation as in~\cite[Section~2]{T09} we introduce some functionals depending on solutions $(r, \psi_r, A_r)$ to the modified Seiberg--Witten equations. We first define the Chern--Simons functional
\[
\fcs:=\int_{M} a_r \wedge da_r\,,
\]
where $a_r:=A_r-A_1$ is a 1-form and $A_1$ is a fixed connection on the bundle $E$ satisfying
\[
2 F_{A_1}+F_{A_0}=0\,.
\]
Recall that the determinant bundle of $E \oplus K^{-1}E$ has torsion first Chern class, so $[2F_{A_1}+F_{A_0}]=0$, and such an $A_1$ always exists. We also define
\[
\fe:=\int_{M} \gamma \wedge da_r\,,
\]
where $\gamma$ is a 1-form so that $d\gamma=i_X \mu$, and
\[
\fe_{\nu}:=\int_{M} (\Gamma+\bom') \wedge da_r\,,
\]
where $\Gamma$ and $\bom'$ are the perturbing 1-forms introduced in Equation~\eqref{bom} and Theorem~\ref{T.existence}, respectively. Finally, we define the action functional
\[
\fa:=\frac{1}{2} \fcs-r\fe-\fe_{\nu}+r \int_{M} \psi^{\dagger}_{r} D_{A_r} \psi_r \,.
\]
One can check that for each fixed $r$, the solutions $(\psi_r,A_r)$ to the Seiberg--Witten equations are critical points of the functional $\fa$.

Finally, we  state a more detailed version of Taubes's Existence Theorem~\ref{T.existence} (see~\cite[Proposition 4.1]{T09}), which ensures that, in fact, the sequence of solutions $(r_n, \psi_n, A_n)$ comes from a piecewise smooth 1-parameter family of solutions $(r, \psi_r, A_r)$:

\begin{proposition}\label{smoothex}{(Taubes~\cite{T09})}
Let $(r_n, \psi_n,A_n)_{n=0}^{\infty}$ be a sequence of solutions to the Seiberg--Witten equations provided by Theorem~\ref{T.existence}. There is an increasing sequence $\{\rho_k\}_{k=1}^{\infty} \subset[1, \infty)$ with no accumulation points so that the following holds:
\begin{enumerate}
\item For each $k$, there is a smooth family of solutions $(r, \tilde{\psi}_r,\tilde{A}_r)$ to the Seiberg--Witten equations parametrized by $r \in (\rho_k, \rho_{k+1})$.
\item The associated functions $\fa(r)$, $\fcs(r)$, $\fe(r)$, $\fe_{\nu}(r)$ and $\cE_r$ defined by this family are smooth on the intervals $(\rho_k, \rho_{k+1})$. Moreover, there is a continuous function $\fa_{0}: [1, \infty) \rightarrow \RR$ such that for any $r \in [1, \infty)\setminus \{\rho_k\}_{k=1}^{\infty}$, $\fa_{0}(r)=\fa(r)$.
\item The sequence $\{r_n\}_{n=0}^{\infty}$ is contained in the set $[1, \infty)\setminus \{\rho_k\}_{k=1}^{\infty}$, and moreover $(\psi_n, A_n)=(\tilde{\psi}_{r_n}, \tilde{A}_{r_n})$ with energy $\cE_n\equiv\cE_{r_n}$.

\end{enumerate}
\end{proposition}

\subsection{Main theorem}

\begin{theorem}\label{T.foliated}
Suppose that the helicity of the vector field $X$ is positive. Let
$(r_{n},\,\psi_n,\,A_n)_{n=1}^\infty$ be a sequence of solutions
to the modified Seiberg--Witten equations as in Proposition~\ref{smoothex}. Assume that the sequence of energies $\cE_n$ is not bounded (i.e., $\liminf_{n\to\infty} \cE_n=\infty$). Then the sequence of $2$-forms
\[
\Si_n:=\frac{F_{A_n}}{\cE_n}
\]
converges, possibly after passing to a subsequence, to a foliation cycle $\Sigma_{\infty}$ of $X$. Moreover, if the one-parameter family of energies $\cE_r$ satisfies the estimate
\begin{equation}\label{eq.assum}
C_1r^\theta\leq \cE_r\leq C_2r^\theta\,,
\end{equation}
for some $\theta\in (0,1)$, some positive constants $C_1,C_2$ and all $r$ large enough, we have
\[
\Sigma_{\infty} \in \overline{\cF_{X}} \cap \cC_{X}\,.
\]
\end{theorem}

We remark that the 1-current $\Sigma_{\infty}$ obtained in this theorem is related to the invariant measure $\sigma_{\infty}$ of Taubes's Theorem~\ref{T.taubes} in the following way: for any 1-form $\theta$ on $M$, we have
\[
\Sigma_{\infty}(\theta)=\sigma_{\infty}(\theta(X))\,.
\]
However,  the proof we give below is different from Taubes's proof of the existence of the invariant measure $\sigma_\infty$.
Additionally, we can interpret the property that $\Sigma_{\infty} \in \overline{\cF_{X}}$ as follows: for any 1-form $\gamma$ satisfying $d \gamma=i_X \mu$, define the function $h_{\gamma}:=\star (\gamma \wedge\, i_{X} \mu)=\gamma(X)$, which is precisely the density of the helicity functional. Then, as argued in the previous section,
$$0=\Sigma_\infty(\gamma)=\sigma_{\infty}(h_{\gamma})\,,$$
which is the refinement stated in Remark~\ref{R.=0}.

\subsection{Proof of Theorem~\ref{T.foliated}}

We divide the proof in three steps. First, we show that $\Si_{n}$ has a subsequence that converges to some non-trivial closed $1$-current $\Sigma_{\infty}$ (this is straightforward). Then, we prove that $\Si_{\infty}$ is a foliation cycle of $X$. Finally, we establish that $\Si_{\infty}(\gamma)=0$ for $d \gamma=i_{X} \mu$, so that $\Si_{\infty}\in \overline{\cF_{X}}\cap \cC_X$.

\subsubsection*{Step 1: $\Sigma_{n} \rightarrow \Sigma_{\infty}$ and $\partial \Sigma_{\infty}=0$}

For any 1-form $\theta$ it is straightforward to see that
\begin{equation}\label{current}
|\Sigma_{n}(\theta)|=|\sigma_n(\theta(X))| \leq C \|\theta\|_{L^{\infty}(M)} \,,
\end{equation}
which implies that the sequence $\Si_{n}$ is bounded in the weak topology. Since in the space of $1$-currents (with the weak topology) bounded subsets are precompact, there is a convergent subsequence. We denote by $\Si_{\infty}$ the limiting $1$-current. It is obvious that $\Sigma_{\infty}$ is not trivial (the zero current) because $\Sigma_{\infty}(\lambda)=1$ (recall that $\la$ is the $1$-form dual to $X$, so $\la(X)=1$).

We observe that the boundary operator $\partial : \cZ^{1}(M) \rightarrow \cZ^{0}(M)$ in the space of $1$-currents is defined by duality as
\[
\partial c(f):=c(df) \,,
\]
and is continuous in the weak topology. Since the curvatures $F_{A_n}$ are closed $2$-forms, we infer that for any $n \in \NN$ and any smooth function $f$
\[
\Sigma_{n}(df)=0\,,
\]
thus implying that $\Sigma_{\infty}(df)=0$, i.e., $\Sigma_{\infty}$ is a closed $1$-current.

\subsubsection*{Step 2: $\Sigma_{\infty} \in \cC_{X}$}

We proceed by contradiction. As is well known, the space of foliation currents $\cZ_{X}$ is a closed convex cone with compact convex base inside the space of $1$-currents. Suppose that $\Sigma_{\infty} \notin \cZ_{X}$. Then, by a standard application of the Hahn--Banach theorem, there is a hyperplane separating $\Sigma_{\infty}$ and $\cZ_{X}$; in other words, there is a continuous linear functional $\cL: \cZ^{1} \rightarrow \RR$ satisfying $\cL(c) \geq 0$ for any $c \in \cZ_{X}$ and $\cL(\Sigma_{\infty}) <0$.

Since the space of $1$-currents and the space of smooth $1$-forms are continuous duals of each other, we can identify the functional $\cL$ with a 1-form $\theta$ satisfying
\[
\theta_p(X) \geq 0 \text{ at any point $p \in M$}\,,
\]
and
\[
\Sigma_{\infty} (\theta) <0 \,.
\]
Thus, to prove that $\Sigma_{\infty}$ is a foliation current it suffices to check that for any $\theta$ with $\theta_p(X)\geq 0$, we must have $\Sigma_{\infty}(\theta) \geq 0$.  Indeed, let $\theta$ be any such 1-form, then
\[
\Sigma_n (\theta)=\frac{\int_M F_{A_n} \wedge \theta}{\cE_n}=\frac{r_n \int_M (1-|\al_n|^2+|\be_n|^2) \theta(X) \mu }{\cE_n}+ \frac{\int_M (r_n\psi_n^{\dagger} \sigma^{\perp} \psi_n + \bom) \wedge \theta }{\cE_n}\,.
\]
Here we have used the notation $\psi^{\dagger} \sigma^{\perp} \psi:= \star (\psi^{\dagger} \sigma \psi)- (\psi^{\dagger} \sigma(X) \psi) i_{X} \mu$. Now, taking an upper bound of the last term in the above equation, we can write
\[
\Sigma_n (\theta) \geq \frac{r_n \int_M (1-|\al_n|^2+|\be_n|^2) \theta(X) \mu }{\cE_n}- \|\theta\|_{L^{\infty}(M)} \frac{\int_M (r_n|\al_n| |\be_n|+C)\mu }{\cE_n} \,.
\]
Using the assumption $\theta(X)\geq 0$ on $M$ and the fact that $r_n(1-|\al_n|^2) \geq -C$ by Lemma~\ref{L.TaubesEstimates}, we obtain
\[
\Sigma_n (\theta)\geq -\frac{\|\theta\|_{L^\infty(M)}}{\cE_n}\big(C + r_n \int_M |\al_n| |\be_n|\mu \big)\,.
\]
Now, applying Lemma~\ref{L.TaubesEstimates} again and using Equation~\eqref{sigma_n}, we can bound the second summand in the above inequality as
\[
r_n \int_{M} |\al_n| |\be_n| \mu \leq C \sqrt{r_n} \int_M |1- |\al_n|^2|^{\frac{1}{2}} \mu \leq C\Big(r_n \int_M |1- |\al_n|^2|\mu\Big)^{\frac{1}{2}}\leq
C\cE_n^{\frac{1}{2}}\,.
\]
We then conclude that
\[
\Sigma_n(\theta)\geq -\frac{\|\theta\|_{L^\infty(M)}}{\cE_n}\big(C + \cE_n^{\frac12}\big)\to 0
\]
as $n\to\infty$ because $\cE_n$ is assumed to be an unbounded sequence, and therefore $\Sigma_\infty(\theta)\geq0$. As we argued before, this implies that $\Sigma_\infty$ is a foliation current and being closed (by Step~$1$) we deduce that it is a foliation cycle, as we wanted to show.

\subsubsection*{Step 3: $\Sigma_{\infty} \in \overline{\cF_X}$}

Let $(r, \tilde\psi_r,\tilde A_r)$ be the 1-parameter family of solutions, whose existence in ensured by Proposition \ref{smoothex}, which coincides with $(r_n, \psi_n,A_n)$ at $r=r_n$. Observe that, since $\cE_n$ is unbounded, we can write
\begin{equation}\label{funs}
\Sigma_{\infty}(\gamma)= \lim_{n \rightarrow \infty} \frac{\fe(r_n)}{\cE_n}\,,
\end{equation}
so $\Sigma_{\infty}\in \overline{\cF_{X}}$ if and only if
\[
\lim_{n \rightarrow \infty} \frac{\fe(r_n)}{\cE_n}=0\,.
\]
Since $\Sigma_{\infty}(\gamma) \leq 0$ by Taubes's Theorem~\ref{T.taubes}, it is enough to prove that there is no constant $C>0$ such that
\begin{equation}\label{Eq.fe}
-\fe(r_n) \geq C \cE_n\,.
\end{equation}

The following lemma is key in what follows.

\begin{lemma}\label{fbounds}
For $r$ large enough, there is a constant $C$ independent of $r$ such that
\[
|\fcs(r)| \leq C r^{\frac{2}{3}} \cE_{r}^{\frac{4}{3}}
\]
\[
|\fe_{\nu}(r)| \leq C \cE_{r}
\]
\end{lemma}
\begin{proof}
The first bound on the $\fcs$ functional follows from the proof of~\cite[Lemma~4.3]{T09} and the assumption that $\cE_r$ is unbounded. As for the second bound, we notice that
\[
\frac{|\fe_{\nu}(r)|}{\cE_r}=|\Sigma_{r}(\Gamma+\bom')|+O(\cE^{-1}_r)
\]
and clearly $|\Sigma_{r}(\Gamma+\bom')|\leq C$ for some constant $C$.
\end{proof}

To show that there is no constant $C>0$ for which Equation~\eqref{Eq.fe} holds, let us assume the contrary. Set $\fb(r):={-2\fa_{0}(r)}/{r}$. By Proposition \ref{smoothex}, the functions $\fb(r)$ and $\fa_{0}(r)$ are differentiable on the intervals $I_k:=(\rho_k, \rho_{k+1})$. Moreover, it is straightforward to check that they verify the relation (cf.~\cite[Section~4]{T09}):
\[
\frac{d\fb}{dr} =\frac{\fcs}{r^2}-\frac{2 \fe_{\nu}}{r^2}\,,
\]
so, in view of Lemma~\ref{fbounds}, we obtain the bound
\[
\Big|\frac{d\fb}{dr}\Big| \leq C \Big(\frac{\cE_{r}}{r}\Big)^{\frac{4}{3}}
\]
on $I_k$, from which we deduce that
\[
|\fb(r)| \leq |\fb(r_0)|+\int_{r_0}^{r}  \Big(\frac{\cE_{\rho}}{\rho}\Big)^{\frac{4}{3}} d \rho
\]
and hence
\begin{equation}\label{fa1}
|\fa_{0}(r)| \leq \frac12r|\fb(r_0)|+\frac12r\int_{r_0}^{r}  \Big(\frac{\cE_{\rho}}{\rho}\Big)^{\frac{4}{3}} d \rho\,.
\end{equation}
Since $\fa_0$ is a continuous function, the inequality~\eqref{fa1} holds for all $r\geq r_0$ with $r,r_0\in[1,\infty)$.

Next, the definition of the functional $\fa(r)$ (notice that $D_{A_r}\psi_r=0$ for solutions of the Seiberg--Witten equations), Lemma~\ref{fbounds}, and our assumption that $-\fe(r_n) \geq C \cE_n$, allow us to write the bound
\[
|\fa_{0}(r_n)| \geq C \bigg(r_n \cE_n-\cE_n-r_n^{\frac{2}{3}} \cE_{n}^{\frac{4}{3}}\bigg) \,.
\]
Since we are assuming that the energy growth is bounded as $\cE_n\leq Cr_n^\theta$ for some $\theta<1$, we easily infer that $ \limsup_{n\to\infty} r_{n}^{\frac{2}{3}} \cE^{\frac{4}{3}}_{n}(r_n \cE_n)^{-1}=0$, and therefore, for large enough $r_n$, the previous bound implies
\begin{equation}\label{fa2}
|\fa_{0}(r_n)| \geq C r_n \cE_n\,.
\end{equation}

Accordingly, Equations~\eqref{fa1} and~\eqref{fa2} hold simultaneously, if
\begin{equation}\label{ineq}
\cE_n \leq C \int_{r_0}^{r_n} \Big(\frac{\cE_{\rho}}{\rho}\Big)^{\frac{4}{3}} d \rho+C
\end{equation}
as $n \rightarrow \infty$. Finally, combining this estimate with the assumption~\eqref{eq.assum}, we derive that
\[
r_n^\theta\leq C(1+r_n^{\frac{4\theta-1}{3}})
\]
provided that $\theta\neq\frac14$, and
\[
r_n^{\frac14}\leq C\ln r_n
\]
when $\theta=\frac14$, for all $n$ large enough. This yields a contradiction with the fact that $\theta\in (0,1)$. We then conclude that there is no constant $C>0$ for which Equation~\eqref{Eq.fe} holds, and hence $\lim_{n\to\infty} \frac{\fe(r_n)}{\cE_n}=0$, which completes the proof of the theorem.

\bibliographystyle{amsplain}

\end{document}